\documentclass[12pt,a4paper]{article}
\usepackage[utf8x]{inputenc}
\usepackage[T1]{fontenc}
\usepackage[a4paper, margin=3cm]{geometry}

\usepackage[pdftex]{graphicx} 
\usepackage{calc} 
\usepackage{enumitem} 

\usepackage[all]{nowidow} 
\usepackage[protrusion=true,expansion=true]{microtype} 

\usepackage{lipsum} 
\usepackage{amsmath}
\usepackage{graphicx}
\usepackage{amssymb}
\usepackage{amsthm}
\usepackage{caption}
\usepackage{bbm}
\usepackage{dsfont}
\usepackage{accents}
\usepackage{mathrsfs}
\usepackage{enumitem}
\usepackage{tikz}
\usepackage{tikz-3dplot}
\usetikzlibrary{patterns}
\usetikzlibrary{arrows.meta}
\usepackage{subcaption}
\usepackage[dvipsnames]{xcolor}
\usetikzlibrary{shapes.misc}
\usetikzlibrary{calc}
\usetikzlibrary{positioning}
\usetikzlibrary{arrows.meta}
\usepackage{booktabs}
\usepackage{float}

\tdplotsetmaincoords{60}{120}

\usepackage[pdftex,linkcolor=black,pdfborder={0 0 0}]{hyperref} 
\numberwithin{equation}{section}

\newtheorem{Def}{Definition}[section]
\newtheorem{Theorem}[Def]{Theorem}
\newtheorem{Lemma}[Def]{Lemma}

\newtheorem{Remark}[Def]{Remark}

\newtheorem{Corollary}[Def]{Corollary}

\newcommand{\N}{\mathbb{N}}
\newcommand{\Z}{\mathbb{Z}}

\DeclareMathOperator{\E}{\mathbb{E}} 
 
\DeclareMathOperator{\supp}{\mathsf{supp}}
\DeclareMathOperator{\mylog}{\mathsf{log}}
\DeclareMathOperator{\myexp}{\mathsf{exp}}

\newcommand{\SigmaA}{
\begin{tikzpicture}[tdplot_main_coords,scale=0.9,line join=round,line cap=round]

    \coordinate (A000) at (0,0,0);
    \coordinate (A100) at (1,0,0);
    \coordinate (A200) at (2,0,0);

    \coordinate (A010) at (0,1,0);
    \coordinate (A110) at (1,1,0);
    \coordinate (A210) at (2,1,0);

    \coordinate (A020) at (0,2,0);
    \coordinate (A120) at (1,2,0);
    \coordinate (A220) at (2,2,0);

    \coordinate (A001) at (0,0,1);
    \coordinate (A101) at (1,0,1);
    \coordinate (A201) at (2,0,1);

    \coordinate (A011) at (0,1,1);
    \coordinate (A111) at (1,1,1);
    \coordinate (A211) at (2,1,1);

    \coordinate (A021) at (0,2,1);
    \coordinate (A121) at (1,2,1);
    \coordinate (A221) at (2,2,1);

    \foreach \A/\B in {
      A000/A100, A100/A200,
      A010/A110, A110/A210,
      A020/A120, A120/A220,
      A000/A010, A010/A020,
      A100/A110, A110/A120,
      A200/A210, A210/A220}
      \draw[very thin,gray!70] (\A)--(\B);

    \foreach \A/\B in {
      A001/A101, A101/A201,
      A011/A111, A111/A211,
      A021/A121, A121/A221,
      A001/A011, A011/A021,
      A101/A111, A111/A121,
      A201/A211, A211/A221}
      \draw[very thin,gray!70] (\A)--(\B);

    \foreach \A/\B in {
      A000/A001, A100/A101, A200/A201,
      A010/A011, A110/A111, A210/A211,
      A020/A021, A120/A121, A220/A221}
      \draw[very thin,gray!70] (\A)--(\B);

    \draw[red,line width=1.2pt] (A110) -- (A111);

  \end{tikzpicture}%
}

\newcommand{\NuA}{%
  \begin{tikzpicture}[tdplot_main_coords,scale=1.0,line join=round,line cap=round]

    \coordinate (O)  at (0,0,0);
    \coordinate (Z)  at (0,0,1);

    \coordinate (Xp) at (1,0,0);
    \coordinate (Xm) at (-1,0,0);
    \coordinate (Yp) at (0,1,0);
    \coordinate (Ym) at (0,-1,0);

    \coordinate (XpZ) at (1,0,1);
    \coordinate (XmZ) at (-1,0,1);
    \coordinate (YpZ) at (0,1,1);
    \coordinate (YmZ) at (0,-1,1);

    \filldraw[fill=teal!30,draw=black,thick]
      (O) -- (Xm) -- (XmZ) -- (Z) -- cycle;

    \filldraw[fill=teal!30,draw=black,thick]
      (O) -- (Ym) -- (YmZ) -- (Z) -- cycle;

    \filldraw[fill=teal!50,draw=black,thick]
      (O) -- (Xp) -- (XpZ) -- (Z) -- cycle;

    \filldraw[fill=teal!50,draw=black,thick]
      (O) -- (Yp) -- (YpZ) -- (Z) -- cycle;

    \draw[red, line width=1pt] (O) -- (Z);

  \end{tikzpicture}%
}

\newcommand{\GraphA}{%
  \begin{tikzpicture}[scale=1.3,every node/.style={circle,inner sep=1pt,draw}]
    \node (p1) at (0,1) {};
    \node (p2) at (-0.5,0.5) {};
    \node (p3) at (0.5,0.5) {};
    \node (p4) at (0,0) {};
    \draw (p1)--(p2)--(p4)--(p3)--(p1);
 
  \end{tikzpicture}%
}

\newcommand{\SigmaB}{
\begin{tikzpicture}[tdplot_main_coords,scale=0.9,line join=round,line cap=round]

    \coordinate (A000) at (0,0,0);
    \coordinate (A100) at (1,0,0);
    \coordinate (A200) at (2,0,0);

    \coordinate (A010) at (0,1,0);
    \coordinate (A110) at (1,1,0);
    \coordinate (A210) at (2,1,0);

    \coordinate (A020) at (0,2,0);
    \coordinate (A120) at (1,2,0);
    \coordinate (A220) at (2,2,0);

    \coordinate (A001) at (0,0,1);
    \coordinate (A101) at (1,0,1);
    \coordinate (A201) at (2,0,1);

    \coordinate (A011) at (0,1,1);
    \coordinate (A111) at (1,1,1);
    \coordinate (A211) at (2,1,1);

    \coordinate (A021) at (0,2,1);
    \coordinate (A121) at (1,2,1);
    \coordinate (A221) at (2,2,1);

    \foreach \A/\B in {
      A000/A100, A100/A200,
      A010/A110, A110/A210,
      A020/A120, A120/A220,
      A000/A010, A010/A020,
      A100/A110, A110/A120,
      A200/A210, A210/A220}
      \draw[very thin,gray!70] (\A)--(\B);

    \foreach \A/\B in {
      A001/A101, A101/A201,
      A011/A111, A111/A211,
      A021/A121, A121/A221,
      A001/A011, A011/A021,
      A101/A111, A111/A121,
      A201/A211, A211/A221}
      \draw[very thin,gray!70] (\A)--(\B);

    \foreach \A/\B in {
      A000/A001, A100/A101, A200/A201,
      A010/A011, A110/A111, A210/A211,
      A020/A021, A120/A121, A220/A221}
      \draw[very thin,gray!70] (\A)--(\B);

    \draw[red,line width=1.2pt] (A110) -- (A111);
    \draw[red,line width=1.2pt] (A120) -- (A121);

  \end{tikzpicture}%
}

\newcommand{\NuB}{%
  \begin{tikzpicture}[tdplot_main_coords,scale=0.8,line join=round,line cap=round]

    \filldraw[fill=teal!30,draw=black,thick]
      (0,0,0) -- (-1,0,0) -- (-1,0,1) -- (0,0,1) -- cycle;

    \filldraw[fill=teal!30,draw=black,thick]
      (0,0,0) -- (0,-1,0) -- (0,-1,1) -- (0,0,1) -- cycle;

    \filldraw[fill=teal!50,draw=black,thick]
      (0,0,0) -- (1,0,0) -- (1,0,1) -- (0,0,1) -- cycle;

    \filldraw[fill=teal!30,draw=black,thick]
      (0,1,0) -- (-1,1,0) -- (-1,1,1) -- (0,1,1) -- cycle;
    \filldraw[fill=teal!50,draw=black,thick]
      (0,1,0) -- (0,2,0) -- (0,2,1) -- (0,1,1) -- cycle;

     \filldraw[fill=teal!50,draw=black,thick]
      (0,1,0) -- (1,1,0) -- (1,1,1) -- (0,1,1) -- cycle;

    \draw[red, line width=1pt] (0,0,0) -- (0,0,1);
\draw[red, line width=1pt] (0,1,0) -- (0,1,1);
  \end{tikzpicture}%
}

\newcommand{\GraphB}{%
  \begin{tikzpicture}[scale=1.1,every node/.style={circle,inner sep=1pt,draw}]
    \node (p1) at (0,1) {};
    \node (p2) at (-0.5,0.5) {};
    \node (p3) at (1,1) {};
    \node (p4) at (0,0) {};
    \node (p5) at (1.5,0.5) {};
    \node (p6) at (1,0) {};
    \draw (p1)--(p2)--(p4)--(p6)--(p5) -- (p3) -- (p1);
 
  \end{tikzpicture}%
}

\newcommand{\SigmaC}{\begin{tikzpicture}[tdplot_main_coords,scale=0.9,line join=round,line cap=round]

    \coordinate (A000) at (0,0,0);
    \coordinate (A100) at (1,0,0);
    \coordinate (A200) at (2,0,0);

    \coordinate (A010) at (0,1,0);
    \coordinate (A110) at (1,1,0);
    \coordinate (A210) at (2,1,0);

    \coordinate (A020) at (0,2,0);
    \coordinate (A120) at (1,2,0);
    \coordinate (A220) at (2,2,0);

    \coordinate (A001) at (0,0,1);
    \coordinate (A101) at (1,0,1);
    \coordinate (A201) at (2,0,1);

    \coordinate (A011) at (0,1,1);
    \coordinate (A111) at (1,1,1);
    \coordinate (A211) at (2,1,1);

    \coordinate (A021) at (0,2,1);
    \coordinate (A121) at (1,2,1);
    \coordinate (A221) at (2,2,1);

    \foreach \A/\B in {
      A000/A100, A100/A200,
      A010/A110, A110/A210,
      A020/A120, A120/A220,
      A000/A010, A010/A020,
      A100/A110, A110/A120,
      A200/A210, A210/A220}
      \draw[very thin,gray!70] (\A)--(\B);

    \foreach \A/\B in {
      A001/A101, A101/A201,
      A011/A111, A111/A211,
      A021/A121, A121/A221,
      A001/A011, A011/A021,
      A101/A111, A111/A121,
      A201/A211, A211/A221}
      \draw[very thin,gray!70] (\A)--(\B);

    \foreach \A/\B in {
      A000/A001, A100/A101, A200/A201,
      A010/A011, A110/A111, A210/A211,
      A020/A021, A120/A121, A220/A221}
      \draw[very thin,gray!70] (\A)--(\B);

    \draw[red,line width=1.2pt] (A110) -- (A111);
    \draw[red,line width=1.2pt] (A111) -- (A121);
  \end{tikzpicture}%
}

\newcommand{\NuC}{%
  \begin{tikzpicture}[tdplot_main_coords,scale=0.8,line join=round,line cap=round]

    \filldraw[fill=teal!30,draw=black,thick]
      (0,0,0) -- (-1,0,0) -- (-1,0,1) -- (0,0,1) -- cycle;

    \filldraw[fill=teal!30,draw=black,thick]
      (0,0,0) -- (0,-1,0) -- (0,-1,1) -- (0,0,1) -- cycle;

    \filldraw[fill=teal!50,draw=black,thick]
      (0,0,0) -- (1,0,0) -- (1,0,1) -- (0,0,1) -- cycle;

    \filldraw[fill=teal!30,draw=black,thick]
      (0,0,1) -- (1,0,1) -- (1,1,1) -- (0,1,1) -- cycle;
   \filldraw[fill=teal!40,draw=black,thick]
   (0,0,1) -- (-1,0,1) -- (-1,1,1) -- (0,1,1) -- cycle;

     \filldraw[fill=teal!50,draw=black,thick]
      (0,0,1) -- (0,0,2) -- (0,1,2) -- (0,1,1) -- cycle;

    \draw[red, line width=1pt] (0,0,0) -- (0,0,1);
\draw[red, line width=1pt] (0,0,1) -- (0,1,1);
  \end{tikzpicture}%
}

\newcommand{\GraphC}{%
  \begin{tikzpicture}[scale=1.1,every node/.style={circle,inner sep=1pt,draw}]
    \node (p1) at (0,1) {};
    \node (p2) at (-0.5,0.5) {};
    \node (p3) at (1,1) {};
    \node (p4) at (0,0) {};
    \node (p5) at (1.5,0.5) {};
    \node (p6) at (1,0) {};
    \draw (p1)--(p2)--(p4)--(p6)--(p5) -- (p3) -- (p1);
 
  \end{tikzpicture}%
}

\newcommand{\SigmaD}{\begin{tikzpicture}[tdplot_main_coords,scale=0.9,line join=round,line cap=round]

    \coordinate (A000) at (0,0,0);
    \coordinate (A100) at (1,0,0);
    \coordinate (A200) at (2,0,0);

    \coordinate (A010) at (0,1,0);
    \coordinate (A110) at (1,1,0);
    \coordinate (A210) at (2,1,0);

    \coordinate (A020) at (0,2,0);
    \coordinate (A120) at (1,2,0);
    \coordinate (A220) at (2,2,0);

    \coordinate (A001) at (0,0,1);
    \coordinate (A101) at (1,0,1);
    \coordinate (A201) at (2,0,1);

    \coordinate (A011) at (0,1,1);
    \coordinate (A111) at (1,1,1);
    \coordinate (A211) at (2,1,1);

    \coordinate (A021) at (0,2,1);
    \coordinate (A121) at (1,2,1);
    \coordinate (A221) at (2,2,1);

    \foreach \A/\B in {
      A000/A100, A100/A200,
      A010/A110, A110/A210,
      A020/A120, A120/A220,
      A000/A010, A010/A020,
      A100/A110, A110/A120,
      A200/A210, A210/A220}
      \draw[very thin,gray!70] (\A)--(\B);

    \foreach \A/\B in {
      A001/A101, A101/A201,
      A011/A111, A111/A211,
      A021/A121, A121/A221,
      A001/A011, A011/A021,
      A101/A111, A111/A121,
      A201/A211, A211/A221}
      \draw[very thin,gray!70] (\A)--(\B);

    \foreach \A/\B in {
      A000/A001, A100/A101, A200/A201,
      A010/A011, A110/A111, A210/A211,
      A020/A021, A120/A121, A220/A221}
      \draw[very thin,gray!70] (\A)--(\B);

    \draw[red,line width=1.2pt] (A110) -- (A111);
    \draw[red,line width=1.2pt] (A200) -- (A201);
  \end{tikzpicture}%
}

\newcommand{\NuD}{%
  \begin{tikzpicture}[tdplot_main_coords,scale=0.8,line join=round,line cap=round]

\filldraw[fill=teal!30,draw=black,thick]
      (-1,0,0) -- (-1,1,0) -- (-1,1,1) -- (-1,0,1) -- cycle;
    \filldraw[fill=teal!30,draw=black,thick]
      (0,0,0) -- (-1,0,0) -- (-1,0,1) -- (0,0,1) -- cycle;

    \filldraw[fill=teal!30,draw=black,thick]
     (0,0,0) -- (0,-1,0) -- (0,-1,1) -- (0,0,1) -- cycle;

    \filldraw[fill=teal!50,draw=black,thick]
      (0,0,0) -- (1,0,0) -- (1,0,1) -- (0,0,1) -- cycle;

\filldraw[fill=teal!50,draw=black,thick]
      (0,0,0) -- (0,1,0) -- (0,1,1) -- (0,0,1) -- cycle;
    \filldraw[fill=teal!30,draw=black,thick]
      (-1,1,0) -- (-2,1,0) -- (-2,1,1) -- (-1,1,1) -- cycle;
    \filldraw[fill=teal!50,draw=black,thick]
      (-1,1,0) -- (-1,2,0) -- (-1,2,1) -- (-1,1,1) -- cycle;

     \filldraw[fill=teal!50,draw=black,thick]
      (-1,1,0) -- (0,1,0) -- (0,1,1) -- (-1,1,1) -- cycle;
     
    \draw[red, line width=1pt] (0,0,0) -- (0,0,1);
\draw[red, line width=1pt] (-1,1,0) -- (-1,1,1);
  \end{tikzpicture}%
}

\newcommand{\GraphD}{%
  \begin{tikzpicture}[scale=1.2,every node/.style={circle,inner sep=1pt,draw}]
\node (p1) at (-0.1464,0.8536) {};
\node (p2) at (-0.8536,0.8536) {};
\node (p3) at (-0.1464,0.1464) {};
\node (p4) at (-0.8536,0.1464) {};

\node (p5) at (0.6464,0.8536) {};
\node (p6) at (0.6464,0.1464) {};
\node (p7) at (1.3536,0.8536) {};
\node (p8) at (1.3536,0.1464) {};
    \draw (p1)--(p2)--(p4)--(p3)--(p1);
    \draw (p5)--(p6)--(p8)--(p7)--(p5);
    \draw (p1) -- (p5) -- (p3) -- (p6) -- (p1);
  \end{tikzpicture}%
}

\title{Ursell functions in lattice gauge theory}
\author{Adrien Malacan \\
\small Mathematical Sciences, Chalmers University of Technology\\
\small \texttt{malacan@chalmers.se}}
\date{}

\begin{document} 
\tdplotsetmaincoords{60}{120}

\maketitle
\begin{abstract}
Ursell functions $U_n$ are higher-order generalizations of the covariance function, which capture the interactions between $n$ random variables. In the classical Ising model, as shown by Shlosman \cite{Shlosman1986SignsOT}, when considering the spins at some locations, the sign of $U_{2n}$ alternates with $n$ and is independent of the locations of the spins considered. In this paper, we study the Ursell function in Ising lattice gauge theory. When the spins at the edges are used as random variables, we show that $U_n$ can be positive, negative, or zero depending on the configuration and the parameter $\beta$. When considering Wilson loops observables as random variables, using the tool of cluster expansion as adapted in \cite{firstpaper}, we prove that at sufficiently low temperature, for any number $n$ of disjoint Wilson loops, there exists a configuration of loops such that the Ursell function $U_n$ is positive. These results contrast sharply with the behavior observed for the Ising model.
\end{abstract} 
\section{Introduction}
\label{sec:introduction}
Lattice gauge theory is a model from statistical mechanics, introduced independently in the 1970s by Wegner \cite{Wegner:1971app} and Wilson \cite{Wilson74}. Wilson's goal was to address the mathematical well-posedness of Yang–Mills theory, which remains an important open problem in probability theory \cite{ymprobabilists}. To achieve this, Wilson proposed a discretization of Euclidean Yang–Mills theory, now known as lattice gauge theory. To define the model, we consider a finite box $B_N:= [-N,N]^d\subset \Z^m$, with $m \geq 2$. We also fix a finite abelian group $G$, which we call the $\textit{gauge group}$. Then, lattice gauge theory is defined as a Gibbs probability measure on the set of 1-forms $\sigma$, which are functions assigning elements from $G$ to the directed edges of $B_N$, denoted by $E(B_N)$, which satisfy,
\begin{equation}
\label{eq:firstproperty}
\sigma(-e)=-\sigma(e) \in G,~ \text{for all} ~ e \in E(B_N).
\end{equation}
A 1-form is called a \textit{gauge field configuration}, and the set of 1-forms in $B_N$ with group $G$ is denoted by $\Omega_1(B_N,G)$. We let $E$ denote the set of directed edges in $\Z^m$, while $E(B_N)$ denotes the directed (or oriented) edges, whose two endpoints are in the box $B_N$.
We call a square made of four oriented edges $e_1,e_2,e_3,e_4$ as in Figure \ref{fig:singleplaquette} an \textit{oriented plaquette}. The set of oriented plaquettes that are contained in $B_N$, i.e., the four oriented edges of the boundary are in $E(B_N)$, is denoted by $P(B_N)$. \newline
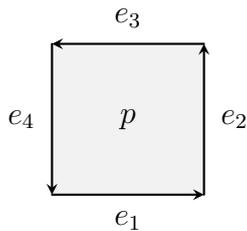
\begin{figure}
\centering
\begin{tikzpicture}[scale=2,>=stealth,line cap=round,line join=round]
  \coordinate (A) at (0,0);
  \coordinate (B) at (1,0);
  \coordinate (C) at (1,1);
  \coordinate (D) at (0,1);

  \fill[gray!10] (A) -- (B) -- (C) -- (D) -- cycle;

  \draw[->,thick] (A) -- node[below,yshift=-2pt] {$e_1$} (B); 
  \draw[->,thick] (B) -- node[right,xshift=2pt] {$e_2$} (C);  
  \draw[->,thick] (C) -- node[above,yshift=2pt] {$e_3$} (D);  
  \draw[->,thick] (D) -- node[left,xshift=-2pt] {$e_4$} (A);  

  \node at (0.5,0.5) {$p$};
\end{tikzpicture}
\caption{A plaquette $p \in P(B_N)$ with it's four oriented boundary edges ${e_1,e_2,e_3,e_4 \in E(B_N)}$.}
\label{fig:singleplaquette}
\end{figure} 
For a 1-form $\sigma \in \Omega_1(B_N,G)$, we define its discrete derivative $d\sigma$ by,
\begin{equation}
\label{def:dderivative}
    d\sigma (p) := \sum_{e \in \partial p} \sigma(e) \in G.
\end{equation}
Here $p \in P(B_N)$, and $\partial p$ denotes the four oriented edges in the boundary of $p$ (see Figure \ref{fig:singleplaquette}). The notion of $\partial p$ will be formalized in Section \ref{subsec:DEC} within the framework of discrete exterior calculus. \newline
We now define the Wilson action functional,
\begin{equation}
\label{eq:Hamiltonian}
S_N(\sigma) := - \sum_{p \in P(B_N)} \Re \rho(d\sigma(p)), ~~ \sigma \in \Omega_1(B_N,G).
\end{equation}
The functional assigns low energy to configurations whose discrete derivative is close to the neutral element, that is, when $\Re \rho(d\sigma(p))$ is close to $1$, where $\rho$ is a unitary representation of $G$.
Finally, the lattice gauge theory's probability measure on gauge field configurations is defined as the following Gibbs measure,
\begin{equation}
\label{eq:Gibbsmeasure}
\mu_{\beta,N}(\sigma) := \frac{1}{Z_{\beta,N}}\myexp(-\beta S_N(\sigma)),
\end{equation}
where $\beta \geq 0$ is the inverse temperature parameter and $Z_{\beta,N}$ the partition function. 
We denote the expectation with respect to this measure by $\E_{N,\beta}$. \newline 
Now, consider a nearest neighbor closed loop in $E(B_N)$, which we denote by $\gamma$. To this loop we associate the Wilson loop's observable $W_\gamma$, defined as the representation in $\rho$ of the sum of the group elements along the edges $e$ traversed by $\gamma$, 
\begin{equation}
\label{eq:defwilsonloopobservable}
W_\gamma(\sigma) := \rho(\sigma(\gamma)) :=\prod_{e \in \gamma} \rho(\sigma(e)) = \rho\big(\sum_{e \in \gamma} \sigma(e) \big).
\end{equation}
We often omit the dependency to $\sigma$ and write $W_\gamma$ for $W_\gamma(\sigma)$.
The infinite volume expectation is denoted by
\begin{equation}
\label{eq:Ginibre}
\E_\beta[W_\gamma ]:= \lim_{N \to \infty} \E_{N,\beta}[W_\gamma],
\end{equation}
which is well-defined by the Ginibre inequalities \cite{Forsstr_m_2023}. When $\beta$ is clear from the context, we write $\E[W_\gamma]$. \newline 
Lattice gauge theory has been extensively studied from a physics perspective, see, e.g., \cite{FrohlichSpencer1982,Guth1980,kogut1979introduction}. In recent years, particularly since the 2020 publication of Chatterjee’s paper \cite{chatterjee2020wilsonloopsisinglattice}, the model has also attracted significant attention from the probability theory community, see, e.g., \cite{adhikari2024correlationdecayfinitelattice,cao2025expandedregimesarealaw, Forsstr_m_2023, garban2021improvedspinwaveestimatewilson}.
An interesting direction is to compare this model with other classical models in statistical mechanics, such as the well-known Ising model. 
In this context, an important tool for understanding the model is the Ursell function (also known as the connected correlation function), which generalizes the notion of covariance to $n \geq 2$ random variables. 
Shlosman \cite{Shlosman1986SignsOT} addressed the question of the sign of Ursell functions in the Ising model, treating spins at different points as random variables. Notably, he proved that the sign of the Ursell function only depends on $n$, and not on the positions of the $n$ points. 
In this paper, we address the corresponding question for Ising lattice gauge theory, meaning $G=\Z_2$. In Theorem~\ref{thm:2} we treat the spins on the edges as random variables, while in Theorem \ref{thm:1} we treat the Wilson loops observables as random variables. Theorem~\ref{thm:2} shows that $U_n$ can take positive, negative, and zero values, depending on the underlying configuration of edges considered, and the value of $\beta$. Theorem \ref{thm:1} states that for every $n$, for all sufficiently large $\beta$, there exist $n$ disjoint Wilson loops $\gamma_1,...,\gamma_n$, such that the corresponding Ursell function is positive. This demonstrates a behavior that contrasts sharply with the result in \cite{Shlosman1986SignsOT}. \newline 
Theorem \ref{thm:1} is the main contribution of this paper, and its proof uses notably the powerful cluster expansion at low temperature, as used for lattice gauge theory by Forsström and Viklund \cite{firstpaper}.
\subsection{Ursell functions}
\label{subsec:Ursellfct}
Ursell functions are tools from statistical mechanics that generalize the notion of covariance to $n$ random variables. 
Formally, they are defined as sums over the set of partitions of $[n] =\{1,...,n\}$, which we denote by $\mathfrak{F}_n$.
 For $\mathfrak{P} \in \mathfrak{F}_n$, we denote by $|\mathfrak{P}|$ the number of elements of the partition. Hence for $\mathfrak{P} \in \mathfrak{F}_n$, one has $\mathfrak{P} = \{P_1,...,P_{|\mathfrak{P}|}\}$, where $\bigcup_{i =1}^{|\mathfrak{P}|} P_i = [n]$ and $(P_i)_{i \in \{1,...,|\mathfrak{P}|\}}$ are pairwise disjoint. Now for $X_1,...,X_n$ random variables on the same space, the Ursell function $U_n(X_1,...,X_n)$ is defined by,
\begin{equation}
\label{eq:Wilson}
    U_n(X_1,...,X_n) := \sum_{\mathfrak{P} \in \mathfrak{F}_n} (-1)^{|\mathfrak{P}|-1} (|\mathfrak{P}|-1)! \prod_{P \in \mathfrak{P}} \E\big[\prod_{i \in P} X_i\big].
\end{equation}
The first Ursell functions look as follows,
\begin{align*}
U_1(X_1)=&\E[X_1], \\
U_2(X_1,X_2)= &\E[X_1X_2]-\E[X_1]\E[X_2] = \mathsf{Cov}(X_1,X_2), \\
U_3(X_1,X_2,X_3)=&\E[X_1X_2X_3]-\E[X_1X_2]\E[X_3]-\E[X_1X_3]\E[X_2] \\  -&\E[X_2X_3]\E[X_1]+2\E[X_1]\E[X_2]\E[X_3].
\end{align*} \newline 
Considering the Ising model, in \cite{Shlosman1986SignsOT} Shlosman proved that in any dimension $m$, when the random variables are the spins $\sigma_{x_i}$ at lattice sites $x_i \in \Z^m$, the sign of the Ursell function only depends on the number $n$ of sites considered and 
not on their specific location. In other words, he proved that,
\[
\mathsf{sign}(U_{2n}(\sigma_{x_1},...,\sigma_{x_{2n}}))=(-1)^{n+1},
\]
while $\mathsf{sign}(U_{2n+1})=0$ due to the $+1/-1$ symmetry of the model. \newline 
For lattice gauge theory, if considering Wilson loop observables as random variables, then the first and second Ursell functions are positive for any values of the parameters $\beta > 0$ and $m \geq 2$. This can be seen as a consequence of Griffith's first and second inequalities. Hence, for any Wilson loops $\gamma_1,\gamma_2$, one has that
\[
U_1(W_{\gamma_1})=\E[W_{\gamma_1}] >0, ~~ U_2(W_{\gamma_1},W_{\gamma_2})=\mathsf{Cov}(W_{\gamma_1},W_{\gamma_2}) > 0.
\]
On the other hand, if considering the spins on edges for $e_1,e_2 \in E$, as a consequence of Elitzur's Theorem \cite{Elitzur}, one has,
\[
U_1(\sigma(e_1))=0, ~~ U_2(\sigma(e_1),\sigma(e_2))=  0.
\]
These results show a behavior that differs from the Ising model and will be further investigated for general values of $n$.
\subsection{Main Results}
In this section, we will present our main results, first for the Ursell function applied to the spins at the edges and then to Wilson observables.
\newline
The random variables observed by Shlosman in the study of the Ursell function for the Ising model \cite{Shlosman1986SignsOT} are the spins at some points. As the value at one point (a 0-form) is not defined in  lattice gauge theory, one natural observable would be to consider 1-forms, and the simplest are the spins at the edges. 
\begin{Theorem}[Ursell function on edges]  
    \label{thm:2}
    Let $\sigma \sim \mu_{N,\beta}$ be a gauge field configuration
    on $\Z^m$, $m \geq 2$, at inverse temperature $\beta > 0$. Then, the following holds. \newline 
    (a) For any $n \in \N$, there exists $e_1,...,e_n \in E$ such that
    \begin{equation*}
        U_{n}(\sigma(e_1),...,\sigma(e_n)) = 0.
    \end{equation*}
    (b) If $n$ is odd or $n=2$, then for any $e_1,...,e_n \in E$ it holds that
    \begin{equation*}
        U_{n}(\sigma(e_1),...,\sigma(e_n)) = 0.
    \end{equation*}
    (c) If $n \ge 4$ is even, then there exist $e_1,...,e_n \in E$ such that
    \begin{align*}
        U_n&(\sigma(e_1),...,\sigma(e_n)) > 0.
    \end{align*}
    \newline 
    (d) Let $m \geq 3$. If $n \geq 10$ is even, and $\beta$ is sufficiently large, then there exist $e_1,...,e_n \in E$ such that
    \begin{align*}
        U_n&(\sigma(e_1),...,\sigma(e_n)) < 0.
    \end{align*}
    Conversely, if $U_n(\sigma(e_1),...,\sigma(e_n)) < 0$, then it must hold that $n \geq 10$. \newline 
    (e) Let $m=2$. If  $n \geq 16$ is even, and $\beta$ is sufficiently large, then there exist $e_1,...,e_n \in E$ such that
    \begin{align*}
        U_n&(\sigma(e_1),...,\sigma(e_n)) < 0.
    \end{align*}
    Conversely, if $U_n(\sigma(e_1),...,\sigma(e_n)) < 0$, then it must hold that $n \geq 16$.
\end{Theorem}
We will see that the proof of Theorem \ref{thm:2} relies on elementary properties of Ising lattice gauge theory, and that most terms in the Ursell function~\eqref{eq:Wilson} typically vanish. The next theorem concerns the case where random variables in the Ursell function are Wilson loop observables, which are natural observables to study since they constitute the main variables of interest in lattice gauge theory. Unlike in the Ising model, there is no symmetry forcing the terms in \eqref{eq:Wilson} to have a value of zero when $n$ is odd. For this reason, it is more delicate to study the sign of $U_n$ in our setting, which is the main reason for employing cluster expansions techniques in Section \ref{sec:Ursellfct} to prove the following Theorem.
\begin{Theorem}[Ursell function on Wilson observables] 
\label{thm:1}
    Let $n \in \N$. Let $\sigma \sim \mu_{N,\beta}$ be a gauge field configuration
    on $\Z^m$, $m \geq 3$, at inverse temperature $\beta > 0$. \newline 
    There exists an inverse temperature $\beta_{n,m}^*$ such that for all $\beta \geq \beta_{n,m}^*$, there are $n$ disjoint Wilson loops  $\gamma_1,...,\gamma_n$ such that
    \begin{equation}
    \label{eq:mainresult}
    U_n(W_{\gamma_1}(\sigma),...,W_{\gamma_n}(\sigma)) > 0.
    \end{equation}
\end{Theorem}
\begin{Remark}
In Theorem \ref{thm:1}, we assume that the loops $\gamma_1,...,\gamma_n$ are disjoint. If this assumption is removed, then $U_n$ can already take negative values for $ n = 3$. To see this, recall that $\E[W_{\gamma}] \in (0,1)$ for any value of $\beta >0$, $m \geq 2$ and any loop $\gamma$. 
Choose any Wilson loop $\gamma_1$ and set $\gamma_1=\gamma_2=\gamma_3$. Then, 
   \[
   U_3(W_{\gamma_1},W_{\gamma_2},W_{\gamma_3}) = 2 \E[W_{\gamma_1}](\E[W_{\gamma_1}]-1) < 0.
   \]
   Similarly, if we choose $\gamma_1=\gamma_2$ and let $\gamma_3$ be any loop, one obtains,
   \[
   U_3(W_{\gamma_1},W_{\gamma_2},W_{\gamma_3}) = -2\E[W_{\gamma_1}] \mathsf{Cov}(W_{\gamma_1},W_{\gamma_3}) < 0,
   \]
   as $\mathsf{Cov}(W_{\gamma_1},W_{\gamma_3}) = U_2(W_{\gamma_1},W_{\gamma_3}) > 0$. \newline
   In the case $n=4,5$, if one takes $\gamma_1=...=\gamma_n$, then
   \begin{align*}
         \begin{cases} U_4(W_{\gamma_1},W_{\gamma_2},W_{\gamma_3},W_{\gamma_4})=-6\E[W_{\gamma_1}]^4+8\E[W_{\gamma_1}]^2-2, \\
          U_5(W_{\gamma_1},W_{\gamma_2},W_{\gamma_3},W_{\gamma_4},W_{\gamma_5}) = 8\E[W_{\gamma_1}](3\E[W_{\gamma_1}]^4-5\E[W_{\gamma_1}]^2+2).
          \end{cases}
   \end{align*}
    From this, it follows that both signs can be obtained for $U_4,U_5$ by varying $\beta$, as  $\E_\beta[W_{\gamma_1}]$ takes values in the neighborhoods of $0$ and $1$ by Lemma \ref{lem:limitbeta} and Lemma \ref{lem:limitbeta1}. One deduces that if the loops $\gamma_1,...,\gamma_n$ are chosen to be equal, then the result of Theorem \ref{thm:1} does not hold.
\end{Remark}
\begin{Remark}
  If we consider the case $m=2$ in Theorem \ref{thm:1}, then one obtains that for any disjoint loops $\gamma_1,...,\gamma_n$ for any value of $\beta > 0$, \[
  U_1(W_{\gamma_1}) > 0, ~~ U_n(W_{\gamma_1},...,W_{\gamma_n}) =0.
  \]
  The first Ursell function being positive is clear, as the expectation of any Wilson observable is positive. The result for $n \geq 2$ follows from the fact that two disjoint Wilson observables have a covariance equal to $0$, e.g., \cite[Section 2.2]{DROUFFE19831}. In that case, one obtains that for any partition $\mathfrak{P} \in \mathfrak{F}_n$,
  \[
  \prod_{P \in \mathfrak{P}} \E[\prod_{i \in P} W_{\gamma_i}] = \prod_{i \in [n]} \E[W_{\gamma_i}].
  \]
  This implies $U_n(W_{\gamma_1},...,W_{\gamma_n}) = 0$.
\end{Remark}
In Section \ref{sec:preliminaries}, we introduce the main notation and background related to lattice gauge theory and cluster expansion. Section \ref{sec:pfthm2} is devoted to proving Theorem \ref{thm:2}. The remaining sections deal with the proof of Theorem \ref{thm:1}. In Section~\ref{sec:Ursellfct}, we analyze the expression of the Ursell function in terms of its cluster expansion representation, and in Section~\ref{sec:Proofmainthm}, we fix the loops and combine the previous results to complete the proof of Theorem~\ref{thm:1}.
\section{Preliminaries}
\label{sec:preliminaries}
In this section, we introduce the main tools we need when working with lattice gauge theory. We also introduce notation in the context of Ursell functions applied to Wilson loops. \newline 
From now on, we will assume that the gauge group is $G=\Z_2$.
\subsection{Discrete Exterior Calculus}
\label{subsec:DEC}
As a crucial tool for working with lattice gauge theory, we will introduce the main ideas and notations related to discrete exterior calculus. We refer to \cite[Section 2]{forsström2021wilsonloopsfiniteabelian} and \cite{Forsstr_m_2023} for a detailed exposition of the topic. \newline 
The central object in discrete exterior calculus is an \textit{oriented k-cell}. For that purpose, we first introduce oriented edges (or oriented 1-cells).  the endpoint of the basis edges in $\Z^m$ as being $\tilde{e}_1:=(1,0,...,0), ..., \tilde{e}_m:=(0,...,0,1)$ and we let $d\tilde{e}_1,...,d\tilde{e}_m$ denote the oriented edges that start at the origin and end in $\tilde{e}_1,...,\tilde{e}_m$. A directed edge $e$ in $\Z^m$ is called \textit{positively oriented} if it equals a translation of one of those unit vectors, i.e., there exists $x \in \Z^m$ and $i \in [m]$ such that $e=x+d\tilde{e}_i$. \newline 
Now, we introduce the wedge product $\wedge$ that satisfies the following properties for any two oriented edges $e_1,e_2$
\[
e_1 \wedge e_2 = - (e_2 \wedge e_1) = (-e_2) \wedge e_1 = e_2 \wedge (-e_1), ~ \text{and} ~ e_1 \wedge e_1 = 0.
\]
We now have the ingredients to define k-cells. Let $e_1, ...,e_k$ be oriented edges such that $e_1 \wedge ...\wedge e_k \neq 0$. Their wedge product $e_1 \wedge ...\wedge e_k$ is then called an \textit{oriented k-cell}. If there is $x \in \Z^m$ such that $e_i = x+d\tilde{e}_{j_{i}}$ for $j_1 < ...<j_k$, then $e_1 \wedge...\wedge e_k$ is called \textit{positively oriented}. If $-(e_1 \wedge ...\wedge e_k)$ is positively oriented, then $e_1 \wedge ...\wedge e_k$ is called \textit{negatively oriented}.
\newline 
The set of k-cells $e_1 \wedge ...\wedge e_k$ contained in $B_N$, i.e., such that all start and endpoints of $(e_i)_{i \in [k]}$ are in $B_N$ is denoted by $C_k(B_N)$. 
We note that $1$-cells are edges, and $2$-cells plaquettes as introduced in Section \ref{sec:introduction}, namely, $P(B_N)=C_2(B_N)$. \newline
The set of positively oriented $k$-cells is denoted by $C_k(B_N)^+$. Thus, $e \in C_1(B_N)^+$ is a positively oriented edge. Similarly, $p \in C_2(B_N)^+$ is a positively oriented plaquette.  \newline
A formal sum of positively oriented $k$-cells $c \in C_k^+(B_N)$ with integer coefficients is called a $k$\textit{-chain}. The set of all $k$-chains is denoted by $C_k(B_N,\Z)$. For $\chi \in C_k(B_N,\Z)$ a k-chain and $y \in C_{k}(B_N)^+$ a positively oriented cell, then $\chi[y] \in \Z$ is the coefficient associated to the cell $y$ in $\chi$. If $\chi \in C_k(B_N,\Z)$ is a k-chain such that for any $y \in C_{k}(B_N)^+$ it holds that $\chi[y] \in \{0,\pm1\}$, then we write the k-chain as a set for convenience: 
\[
\begin{cases}
    y \in \chi, ~ \text{if} ~~ \chi[y]=1, \\
    -y \in \chi,  ~ \text{if} ~ ~ \chi[y]=-1.
\end{cases}
\]
\newline 
Now fix some $c := e_1 \wedge ... \wedge e_k \in C_k(B_N)$, with $k \geq 2$ and $e_i = x + d\tilde{e}_{j_{i}}$ for $x \in \Z^m$. One then defines the \textit{boundary} of $c$, denoted $\partial c \in C_{k-1}(B_N,\Z)$ as the ($k-1$)-chain
\begin{align*}
\partial c := \sum_{i=1}^k & (-1)^{i+1} \big( (e_1 \wedge ... \wedge e_{i-1} \wedge e_{i+1} \wedge ... \wedge e_{k} ) \\ - \big(y_r + &d\tilde{e}_{j_{1}} \wedge ... \wedge y_r+ d\tilde{e}_{j_{i-1}} \wedge y_r+ d\tilde{e}_{j_{i+1}} \wedge ... \wedge y_r+ d\tilde{e}_{j_{k}} )  \big),
\end{align*}
where $y_r := x + d\tilde{e}_{j_r}$. Similarly, if $c \in C_k(B_N)$ and $k \leq m-1$, then one defines the \textit{coboundary} of $c$, denoted by $\hat{\partial}c$, as the $(k+1)$-chain
\[
\hat{\partial}c=\sum_{c': c \in \partial c'} c'.
\] 
In particular, for $p \in C_2(B_N)$, $\partial p \in C_1(B_N,\Z)$ is the formal sum of the four edges in its boundary and $\hat{\partial}p$ is the formal sum of the 3-cells (whose numbers depend on the dimension $m$) in which $p$ is in the boundary. This is shown in Figure \ref{fig:boundaryandcob}. 
\begin{figure}[h]
\centering

\begin{subfigure}[t]{0.45\textwidth}
\centering
\begin{tikzpicture}[scale=2,>=stealth,line cap=round,line join=round]
  \coordinate (A) at (0,0);
  \coordinate (B) at (1,0);
  \coordinate (C) at (1,1);
  \coordinate (D) at (0,1);

  \fill[red!20] (A) -- (B) -- (C) -- (D) -- cycle;

  \draw[->,thick,red!50!black] (A) -- node[below,yshift=-2pt] {$e_1$} (B);
  \draw[->,thick,red!50!black] (B) -- node[right,xshift=2pt] {$e_2$} (C);
  \draw[->,thick,red!50!black] (C) -- node[above,yshift=2pt] {$e_3$} (D);
  \draw[->,thick,red!50!black] (D) -- node[left,xshift=-2pt] {$e_4$} (A);

  \node at (0.8,0.2) {$p$};
  \draw[-{Stealth[length=2.4mm,width=1.8mm]},shorten >=0.6pt,red!70!black]
    (0.5,0.5) ++(0.23,0) arc (0:300:0.23);
\end{tikzpicture}
\caption{Plaquette $p \in C_2(B_N)^+$ with \mbox{$\partial p=e_1+e_2+e_3+e_4$}.}
\end{subfigure}
\hspace{0.05\textwidth}
\begin{subfigure}[t]{0.45\textwidth}
\centering
\begin{tikzpicture}[scale=2,>=stealth,line cap=round,line join=round]
  \coordinate (A) at (0,0);
  \coordinate (B) at (1,0);
  \coordinate (C) at (1,1);
  \coordinate (D) at (0,1);

  \coordinate (A') at ($(A)+(0.40,0.30)$);
  \coordinate (B') at ($(B)+(0.40,0.30)$);
  \coordinate (C') at ($(C)+(0.40,0.30)$);
  \coordinate (D') at ($(D)+(0.40,0.30)$);

  \coordinate (A'') at ($(A)+(-0.40,-0.30)$);
  \coordinate (B'') at ($(B)+(-0.40,-0.30)$);
  \coordinate (C'') at ($(C)+(-0.40,-0.30)$);
  \coordinate (D'') at ($(D)+(-0.40,-0.30)$);

  \fill[red!20] (A)--(B)--(C)--(D)--cycle;
  \draw[thick,red!50!black] (A)--(B)--(C)--(D)--cycle;
  \node at (0.9,0.1) {$p$};

  \draw[gray] (A')--(B')--(C')--(D')--cycle;
  \draw[gray] (A)--(A') (B)--(B') (C)--(C') (D)--(D');
  \node[gray!60!black] at (1.3,1.1) {$c_1$};

  \draw[gray] (A'')--(B'')--(C'')--(D'')--cycle;
  \draw[gray] (A)--(A'') (B)--(B'') (C)--(C'') (D)--(D'');
  \node[gray!60!black] at (-0.3,0.6) {$c_2$};

  \draw[-{Stealth[length=2.4mm,width=1.8mm]},shorten >=0.6pt,red!70!black]
    (0.5,0.5) ++(0.2,0) arc (0:280:0.15);

\end{tikzpicture}
\caption{Same plaquette in $\mathbb{Z}^3$, with the 3-cells $c_1,c_2 \in \hat{\partial}p$.}
\end{subfigure}

\caption{A plaquette $p$ with its boundary on Figure (a) and coboundary on Figure (b) in the case $m=3$.}
\label{fig:boundaryandcob}
\end{figure}
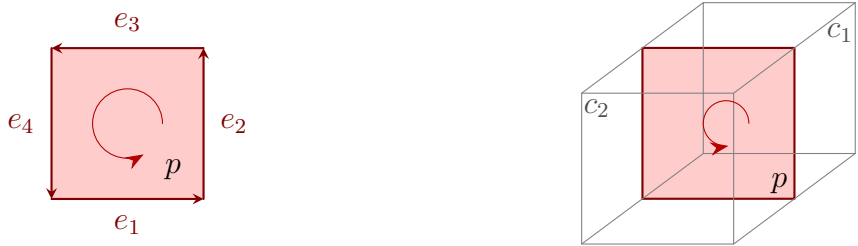
\newline 
A function $f : C_k(B_N) \to G$ is called a $k-$\textit{form} if $f(-c)=-f(c) \in G$ for all $c \in C_k(B_N)$. The set of $k$-forms is denoted by $\Omega_k(B_N,G)$. The support of a $k$-form is denoted by $\supp ~f := \{c \in C_k(B_N) : f(c) \neq 0\}$ while the support on positively oriented cells is written as $\supp ~f^+ := \{c \in C_k^+(B_N) : f(c) \neq 0\}$. \newline 
For $k \leq m-1$ one defines the \textit{exterior derivative}, denoted $d$, as the following mapping,
\[
d : \Omega_k(B_N,G) \to \Omega_{k+1}(B_N,G), ~(d f)(y) := \sum_{c \in \partial y}f(c), ~~ y \in C_{k+1}(B_N).
\]
A 1-chain $\gamma \in C_1(B_N,\Z)$ that has connected support and satisfies $\gamma(e) \in \{-1,0,1\}$ for all $e \in C_1(B_N)$ and $\partial \gamma = 0$ is called a \textit{loop}.  
A 2-chain $q \in C_2(B_N,\Z)$, where $q(p) \in \{-1,0,1\} ~ \text{for all} ~  p \in C_2(B_N)$ and $\partial q = \gamma$, is called an \textit{oriented surface with boundary $\gamma$}. \newline 
By the discrete version of \textit{Stokes' theorem}, for a one-form $\sigma \in \Omega_1(B_N,G)$ and any oriented surface $q$ with boundary $\gamma$, one has
\begin{equation}
\label{eq:Stokesthm}
    \sum_{e\in \gamma} \sigma(e) =  \sum_{p \in q} d\sigma(p).
\end{equation}
 In particular, if $q,q' \in C_2(B_N,\Z)$ have the same boundary, then 
 \[
 \sum_{p \in q} d\sigma(p)=\sum_{p \in q'} d\sigma(p)
 \]
by Stokes' Theorem. \newline
If $f \in \Omega_k(B_N,G)$, then $f$ is said to be \textit{closed} if $df = 0$. The set of closed $k$-forms is denoted by $\Omega_k^0(B_N,G)$.
Another key result in discrete exterior calculus is the \textit{discrete Poincaré lemma} \cite[Lemma 2.2]{chatterjee2020wilsonloopsisinglattice}, which states that for any  $\nu \in \Omega_2^0(B_N,G)$ there exists a 1-form $\sigma \in \Omega_1(B_N,G)$ such that $d\sigma=\nu$. Moreover, the number of such $\sigma$ only depends on the group $G$ and the box $B_N$, not on the specific choice of $\nu$.
\subsection{Cluster expansions at low temperature}
\label{subsec:ClusterExpansion}
Cluster expansions are a widely used tool in statistical mechanics, which allows representing a partition function as a series of weights of \textit{clusters}. We refer to \cite[Chapter 5]{stmechanics} for an introduction to cluster expansions, as well as \cite{ClusteronLGT82} for a discussion on its applications to lattice gauge theory. We will here be interested in the cluster expansion at low temperature, i.e., for $\beta$ large, and will follow the notation of \cite{firstpaper}. In~ this section, we will focus primarily on the notation and the key aspects required for the proofs presented later in the paper.
Let us recall some concepts and notation. \newline 
One defines the adjacency graph $\mathcal{G}_2$ as the graph having vertex set $C_2(B_N)^+$, with an edge between $p_1,p_2 \in C_2(B_N)^+$ if their coboundaries intersect, i.e., if $\supp ~\hat{\partial}p_1 ~\cap ~ \supp ~\hat{\partial}p_2 \neq \emptyset$. In other words, $p_1,p_2$ are connected in the graph $\mathcal{G}_2$, denoted by $p_1 \sim p_2$, if  there exists a 3-cell $c \in C_3(B_N)$ that contains both $p_1,p_2$ (up to orientation) in its boundary. \newline 
We call a closed 2-form $\nu \in \Omega_2^0(B_N,\Z_2)$ that has no support on the boundary of $B_N$ a \textit{vortex} if $\supp \nu^+$ induces a connected subgraph of $\mathcal{G}_2$. The set of vortices is denoted by $\Lambda(B_N)$. For two vortices $\nu_1,\nu_2 \in \Lambda(B_N)$, we write $\nu_1 \sim \nu_2$ if there exist some plaquettes $p_1 \in \supp ~ \nu_1^+,~ p_2 \in \supp ~ \nu_2^+$ such that $p_1 \sim p_2$ in the graph $\mathcal{G}_2$. In that case we call $\nu_1$ and $\nu_2$ \textit{connected}. Then, for a two-chain $q \in C_2(B_N,\Z)$, one defines
\begin{equation}
\label{eq:sumtwochains}
\nu(q) := \sum_{p \in q} \nu(p),
\end{equation}
which is a sum in $\Z_2$.
Now consider a multiset composed of different vortices in $\Lambda(B_N)$,
\[
\mathcal{V} = \{\nu_1,...,\nu_1,\nu_2,...,\nu_2,...,\nu_k,...,\nu_k\}.
\]
We call $\mathcal{V}$ a \textit{vortex cluster} if for any partition of two sets of $\mathcal{V}$ denoted by $\mathcal{V}_1,\mathcal{V}_2 \subset \mathcal{V}$, there is some $(\nu_1,\nu_2) \in \mathcal{V}_1 \times \mathcal{V}_2$ such that $\nu_1 \sim \nu_2$. In simpler terms, we can't divide $\mathcal{V}$ into two multisets that aren't connected. The set of all vortex clusters contained in $B_N$ is denoted by $\Xi_N$, which we will denote by $\Xi$ in this paper. \newline 
For a vortex cluster $\mathcal{V}$ we define,
\begin{equation}
\label{eq:notation||}
|\mathcal{V}| = \sum_{\nu \in \mathcal{V}} n_\mathcal{V}(\nu)|\supp  \nu^+|, ~~ \supp(\mathcal{V}) = \bigcup_{\nu \in \mathcal{V}} \supp  \nu ~~ \text{and} ~~ n(\mathcal{V}) =  \sum_{\nu \in \mathcal{V}} n_\mathcal{V}(\nu),
\end{equation}
where for $\nu \in \Lambda(B_N)$, $n_\mathcal{V}(\nu)$ denotes the number of times the vortex $\nu$ appears in the multiset $\mathcal{V}$. \newline 
If $q \in C_2(B_N,\Z)$, we define,
\begin{equation}
\label{eq:chain}
\mathcal{V}(q):=\sum_{\nu \in \mathcal{V}} \nu(q),
\end{equation}
which is a sum in $\Z_2$ over the multiset $\mathcal{V}$. If $\mathcal{V}(q)=1$, we say that $\mathcal{V}$ is \textit{interacting} with $q$. \newline 
We now define the \textit{activity} of a 2-form $\nu \in \Omega_2(B_N,\Z_2)$, by 
\begin{align*}
\phi_\beta(\nu) &:= \myexp\Big(\sum_{p \in C_2(B_N)} -2\beta \nu(p)\Big) \\ &= \myexp(-2\beta |\supp ~ \nu|) = \myexp(-4\beta |\supp ~ \nu^+|).
\end{align*}
This definition can be extended to vortex clusters $\mathcal{V} \in \Xi$, by letting
\[
\phi_\beta(\mathcal{V}) := \prod_{\nu \in \mathcal{V}} \phi_\beta(\nu)^{n_\mathcal{V}(\nu)}=\myexp(-4\beta|\mathcal{V}|).
\]
From the definition of the activity, it follows that the Gibbs measure in \eqref{eq:Gibbsmeasure} equals
\begin{equation}
\label{eq:partfct1}
\mu_{N,\beta}(\sigma) = \frac{\phi_\beta(d\sigma)}{\hat{Z}_{\beta,N}},
\end{equation}
where the partition function equals
\begin{equation}
\label{eq:defoftildeZ}
\hat{Z}_{\beta,N} := \sum_{\sigma \in \Omega_1(B_N,\Z_2)} \phi_\beta(d\sigma).
\end{equation}
Similarly, one deduces an expression for the Gibbs measure on vortices $\nu \in \Lambda(B_N)$,
\begin{equation}
\label{eq:newpartfct1}
\mu_{N,\beta}(\{\sigma: d\sigma=\nu\}) = \sum_{\sigma: d\sigma=\nu} \frac{\phi_\beta(\nu)}{\hat{Z}_{\beta,N}}=\frac{|\{\sigma: d\sigma=\nu\}|\phi_\beta(\nu)}{\hat{Z}_{\beta,N}}=: \frac{\phi_\beta(\nu)}{\tilde{Z}_{\beta,N}}.
\end{equation}
Here, the first equality comes from \eqref{eq:partfct1} and the last equality is well-defined as a consequence of Poincar\'e's lemma : $|\{\sigma: d\sigma=\nu\}|$ is independent of $\nu$. \newline 
The idea of cluster expansions is to write the logarithm of the partition function $\tilde{Z}_{\beta,N}$ as a sum over weights of vortex clusters $\mathcal{V} \in \Xi$. 
To illustrate this, we define the following. Let $k \geq 1$ and denote by $G^k$ the set of undirected connected graphs with vertex set $\{1,...,k\}$. For $G \in G^k$, we denote the set of edges by $E(G)$. Now, for any $\nu_1,...,\nu_k \in \Lambda(B_N)$, define
\begin{equation}
\label{eq:DefUrsellsecond}
\mathcal{U}(\nu_1,...,\nu_k) := \frac{1}{k!} \sum_{G \in \mathcal{G}^k} (-1)^{|E(G)|} \prod_{(i,j) \in E(G)} \mathbbm{1}\{\nu_i \sim \nu_j\}.
\end{equation}
For $\mathcal{V} \in \Xi$ such that $n(\mathcal{V})=k$, containing $\nu_1,...,\nu_k$ (with multiplicities), we define
\[
\mathcal{U}(\mathcal{V}):=k!~ \mathcal{U}(\nu_1,...,\nu_k).
\]
Now we define the \textit{correlated activity} of the vortex cluster $\mathcal{V} \in \Xi$ as
\[
\Psi_\beta(\mathcal{V}) := \mathcal{U}(\mathcal{V})\phi_\beta(\mathcal{V}).
\]
From this, it follows that
\begin{equation}
\label{eq:Decompositioncorracti}
    \Psi_\beta(\mathcal{V})=\mathcal{U}(\mathcal{V})\myexp\big(-4\beta \sum_{\nu \in \mathcal{V}}n_\mathcal{V}(\nu)|\supp ~ \nu^+|\big)=\mathcal{U}(\mathcal{V})\myexp(-4\beta|\mathcal{V}|). 
\end{equation}
Using this notation, \cite[Lemma 3.3]{firstpaper} proves that at temperature low enough, for $\mathcal{V} \in \Xi$,
\begin{equation}
\label{eq:partfct3}
    \mylog(\tilde{Z}_{\beta,N}) = \sum_{\mathcal{V}\in \Xi} \Psi_\beta(\mathcal{V}).
\end{equation}
At this point, we remark that, unlike the sum in \eqref{eq:defoftildeZ}, which is a finite sum, the sum over vortex clusters \eqref{eq:partfct3} is an infinite series. Hence, the delicate part in applying cluster expansions is to ensure that $\beta$ is large enough for the sum to be absolutely convergent. This is done in \cite[Lemma 3.2]{firstpaper}, where the authors use the fact that when $\beta$ is large enough, the series over correlated activities of $\mathcal{V}$ containing some vortex $\nu \in \Lambda(B_N)$ is bounded by a function depending on $\supp \nu$. We will denote the smallest value that makes \eqref{eq:partfct3} absolutely convergent by $\beta_0^m > 0$, which we will write $\beta_0$ when $m$ is clear from the context. Hence, for $\beta \geq \beta_0^m$, the cluster expansion is well defined for Ising lattice gauge theory on $\Z^m$. \newline 
We can use \eqref{eq:partfct3} to write the expectation of Wilson loops. Let $q \in C_2(B_N,\Z)$ be an oriented surface contained in $B_N$ such that $\partial q=\gamma$ for some loop $\gamma$. Then, by \cite[Proposition 3.5]{firstpaper} we obtain
\begin{equation}
    \label{eq:expwilsonloop}
    -\mylog\big(\E_{\beta,N}[W_\gamma]\big)=\sum_{\mathcal{V} \in \Xi} \Psi_\beta(\mathcal{V}) \Big(1- \rho\big(\mathcal{V}(q)\big)\Big) = \sum_{\mathcal{V} \in \Xi} 2 \mathcal{V}(q)\Psi_\beta(\mathcal{V}),
\end{equation}
where the second equality follows from the fact that the gauge group is $\Z_2$. 
Furthermore, \cite[Theorem 1.2]{firstpaper} proves that for any loop $\gamma \in C_1(B_N,\Z)$, the following holds
\begin{equation}
\label{eq:firstorderexpectation}
\mylog\big(\E_{N,\beta}[W_\gamma]\big) = |\gamma|\big(\myexp(-4(2(m-1)\beta)\big)+o(\myexp(-4(2(m-1)\beta)),
\end{equation}
where the rest term $o(\myexp(-4(2(m-1)\beta)))$ can be made independent of $N$.
\subsection{Limit Results}
In this section, we will look at Wilson loops expectation when $\beta \to 0$ or $\beta \to \infty$. Lemma \ref{lem:limitbeta} will be useful in the next section to prove Theorem \ref{thm:2}.
Recall that the infinite volume expectation is taken as a limit when the box $B_N$ becomes larger \eqref{eq:Ginibre}.
\begin{Lemma}
\label{lem:limitbeta}
   Consider Ising lattice gauge theory  on the lattice $\Z^m$, $m \geq 3$. Let $\gamma$ be any closed loop. Then,
   \[
   \lim_{\beta \to \infty} \E_\beta[ W_\gamma] = 1.
   \]
\end{Lemma}
  \begin{proof}
Apply \eqref{eq:firstorderexpectation} and let $\beta \to \infty$. The claim follows as the term $\text{o}\!\left(e^{-4(2(m-1)\beta)}\right)$ can be made independent of $N$.
  \end{proof}
\begin{Lemma}
\label{lem:limitbeta1}
    Consider Ising lattice gauge theory  at high temperature on the lattice $\Z^m, ~ m\geq 3$, then the following holds for any loop $\gamma$,
    \[
    \lim_{\beta \to 0} \E_\beta[ W_\gamma]= 0.
    \]
\end{Lemma}
\begin{proof}
Using current expansions \cite[Proposition 5.5]{currentexpansion} proves that for any loop $\gamma$, if $N$ is large enough and $\beta$ small enough, then
\[
\E_{N,\beta}[W_\gamma] \leq \frac{(4(m-1)\beta)^{\text{area}(\gamma)}}{1-4(m-1)\beta}.
\]
Since the term $\operatorname{area}(\gamma)$ depends only on $\gamma$ and is positive, and since the right-hand side does not depend on $N$, the conclusion follows. 
\end{proof}
In the two-dimensional case, Ising lattice gauge theory is exactly solvable. This can be done using the technique of gauge fixing. Using that, one can then prove that for a closed loop $\gamma$ it holds that,
\begin{equation*}
\E_\beta[W_\gamma]=\tanh{(\beta)}^{\text{area}(\gamma)}.
\end{equation*}
This yields the limits,
\begin{equation}
\label{eq:limto1whenm2}
\lim_{\beta \to 0} \E_\beta[W_\gamma] = 0 ~~ \text{and} ~ \lim_{\beta \to \infty} \E_\beta[W_\gamma] = 1.
\end{equation}
\section{Proof of Theorem \ref{thm:2}}
\label{sec:pfthm2}
In this section, we consider the Ursell function for spins on edges $e \in E$, where $E$ is the set of directed edges in $\Z^m$ as defined in Section \ref{sec:introduction}. Our goal is to prove Theorem~\ref{thm:2}. For $e_1,...,e_n \in E$, we see from \eqref{eq:Wilson}, that the Ursell function takes the form,
\begin{equation}
\label{eq:Ursellonedges}
    U_n(\sigma(e_1),...,\sigma(e_n))= \sum_{\mathfrak{P} \in \mathfrak{F}_n} (-1)^{|\mathfrak{P}|-1} (|\mathfrak{P}|-1)! \prod_{P \in \mathfrak{P}} \E\big[\prod_{i \in P} \sigma(e_i)\big].
\end{equation}
If the formal sum of $(e_i)_{i \in I}$, with $I \subset [n]$, forms a loop denoted $\gamma$, we write in this section $U_n^{\sigma}(W_\gamma) := U_n(\sigma(e_1),...,\sigma(e_n))$. To prove Theorem \ref{thm:2}, we will use the following lemma.
\begin{Lemma}
\label{lem:Urselledges}
    Let $e_1,...,e_n \in E$ and let $\mathfrak{P} \in \mathfrak{F}_n$ be a partition. Assume that $\sigma \sim \mu_{N,\beta}$ with $\beta >0$.
    Then,
    \[
   \prod_{P \in \mathfrak{P}} \E\big[\prod_{i \in P} \sigma(e_i)\big] \neq 0 \iff \forall P\in\mathfrak{P} : \{e_i\}_{i \in P} ~~ \text{forms a loop} .
    \]
    In words, the expectation of the product of the spins vanishes unless the edges form a loop.
\end{Lemma}
\begin{proof} 
    Assume first that the formal sum of $\{e_i\}_{i \in P}$ forms a loop for all $P \in \mathfrak{P}$. Then, for any $P \in \mathfrak{P}$, there exists a loop $\gamma_P = \sum_{i \in P} e_i$. Now, for this loop, we have
    \[
    \E\big[\prod_{i \in P} \sigma(e_i)\big]  = \E[W_{\gamma_P}] > 0,
    \]
    by, e.g., \cite[Proposition 6.2]{currentexpansion}. 
    Conversely, the product of edges spins along a non-closed path is not invariant under gauge symmetry. Specifically, $\prod_{i \in P} \sigma(e_i)$ is not invariant under gauge transforms at the endpoints. By Elitzur's Theorem \cite{Elitzur}, its expectation must vanish,
    \[
    \E\big[\prod_{i \in P} \sigma(e_i)\big]  = 0.
    \]
    This concludes the proof.
\end{proof}

 As one can intuitively think, Lemma \ref{lem:Urselledges} implies that most of the summands in \eqref{eq:Ursellonedges} will take the value zero, which will facilitate the proof of Theorem \ref{thm:2}. \newline 
 We now see a second Lemma which will be useful in the proof of Theorem \ref{thm:2}(e).
\begin{Lemma}
    \label{lem:abouttheloopdec}
    Let $\gamma$ be a closed loop in $\Z^2$. Assume that $\gamma$ can be decomposed as formal sum of two different loops $\gamma_1,\gamma_2 \neq \gamma_1',\gamma_2'$:
    \[
    \gamma = \gamma_1+\gamma_2 = \gamma_1'+\gamma_2'.
    \]
    Then $\gamma$ must satisfy $|\gamma| \ge 16$, i.e., $\gamma$ contains at least 16 edges.
\end{Lemma}
The proof of Lemma \ref{lem:abouttheloopdec} is referred to Appendix \ref{App:A}. We now have the ingredients to prove Theorem \ref{thm:2}.
\begin{proof}[Proof of Theorem \ref{thm:2}]
(a) Choose $e_1,...,e_n \in E$ without shared boundaries, i.e., $(e_i)_{i \in [n]}$ with $\partial e_i \cap \partial e_j = \emptyset, ~ \text{for} ~ i,j \in [n], i \neq j$. Then, by Lemma \ref{lem:Urselledges}, we have $U_n(\sigma(e_1),...,\sigma(e_n)) = 0$. This completes the proof in this case. \newline
(b) Note that if $n$ is odd, any partition $\mathfrak{P} = \{P_1,...,P_{\mathfrak{|p|}}\} \in \mathfrak{F}_n$ will contain at least one element $P_i \subset [n]$ of odd cardinality. As an odd number of edges cannot form a closed path, Lemma \ref{lem:Urselledges} implies that the Ursell function takes the value zero. In the case $n=2$, the conclusion follows immediately from Lemma \ref{lem:Urselledges}. This concludes the second case.
\newline
    (c) Now assume that $e_1,...,e_n \in E$ form a rectangular loop, which we denote by~$\gamma$. Observe that any strict subset $I \subsetneq [n]$ of the edges breaks the loop, so that by Lemma \ref{lem:Urselledges},
    \[
    \E[\prod_{i \in I} \sigma(e_i)]=0.
    \]
    It follows that in this case, we have
    \[
    U_n(\sigma(e_1),...,\sigma(e_n)) = \E[\prod_{i \in [n]}\sigma(e_n)] = \E[W_\gamma] > 0.
    \]
    This concludes the proof. \newline 
(d)  Now, assume that $m \geq 3$, $n = 10$, and let $\gamma$ be the loop described in Figure \ref{fig:u10negative}. We denote the edges appearing in the formal sum in $\gamma$ by $e_1,...,e_{10}$. Our goal is to show that for certain values of $\beta >0$ one has $U_{10}^{\sigma}(W_\gamma) < 0$.
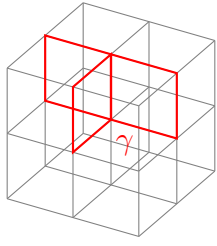
\begin{figure}[ht]
\centering

\begin{tikzpicture}[tdplot_main_coords, scale=1]


  \foreach \x in {0,1} {
    \foreach \y in {0,1,2} {
      \foreach \z in {0,1,2} {
        \draw[gray] (\x,\y,\z) -- (\x+1,\y,\z);
      }
    }
  }

  \foreach \x in {0,1,2} {
    \foreach \y in {0,1} {
      \foreach \z in {0,1,2} {
        \draw[gray] (\x,\y,\z) -- (\x,\y+1,\z);
      }
    }
  }

  \foreach \x in {0,1,2} {
    \foreach \y in {0,1,2} {
      \foreach \z in {0,1} {
        \draw[gray] (\x,\y,\z) -- (\x,\y,\z+1);
      }
    }
  }

  \draw[red, thick] (1,1,1) -- (1,2,1) -- (1,2,2) -- (1,1,2) -- cycle;
  \draw[red, thick] (1,1,1) -- (2,1,1) -- (2,1,2) -- (1,1,2) -- cycle;
  \draw[red, thick] (1,1,1) -- (1,0,1) -- (1,0,2) -- (1,1,2) -- cycle;
  \node[red] at (1.5,1.5,1) {$\gamma$};

\end{tikzpicture}

\caption{The loop $\gamma$, s.t. $U_{10}^{\sigma}(W_\gamma) < 0$ for $\beta$ large.}
\label{fig:u10negative}
\end{figure}
\newline 
    To proceed, we observe that $\gamma$ can be divided into a sum of two loops in exactly three distinct ways (see Figure \ref{fig:Partitonsu10gamma}): in each such sum, one loop has size $4$ and the other loop has size $6$. Specifically we denote these partitions by $\gamma= \gamma_1'+\gamma_2'$, $\gamma= \gamma''_1+\gamma_2''$ and $\gamma = \gamma'''_1+\gamma_2'''$. 
\begin{figure}[ht]
\centering
\begin{minipage}{0.3\textwidth}
\centering
\begin{tikzpicture}[tdplot_main_coords, scale=1]


  \foreach \x in {0,1} {
    \foreach \y in {0,1,2} {
      \foreach \z in {0,1,2} {
        \draw[gray] (\x,\y,\z) -- (\x+1,\y,\z);
      }
    }
  }

  \foreach \x in {0,1,2} {
    \foreach \y in {0,1} {
      \foreach \z in {0,1,2} {
        \draw[gray] (\x,\y,\z) -- (\x,\y+1,\z);
      }
    }
  }

  \foreach \x in {0,1,2} {
    \foreach \y in {0,1,2} {
      \foreach \z in {0,1} {
        \draw[gray] (\x,\y,\z) -- (\x,\y,\z+1);
      }
    }
  }

  \draw[blue, thick] (1,1,1) -- (2,1,1) -- (2,1,2) -- (1,1,2) -- cycle;
  \draw[blue, thick] (1,1,1) -- (1,0,1) -- (1,0,2) -- (1,1,2) -- cycle;
  \draw[red, thick] (1,1,1) -- (1,2,1) -- (1,2,2) -- (1,1,2) -- cycle;
  \node[red] at (1.5,1.5,1) {$\gamma_1'$};
  \node[blue] at (1.5,0.5,1) {$\gamma_2'$};
\end{tikzpicture}
\end{minipage}
\begin{minipage}{0.3\textwidth}
\centering
\begin{tikzpicture}[tdplot_main_coords, scale=1]


  \foreach \x in {0,1} {
    \foreach \y in {0,1,2} {
      \foreach \z in {0,1,2} {
        \draw[gray] (\x,\y,\z) -- (\x+1,\y,\z);
      }
    }
  }

  \foreach \x in {0,1,2} {
    \foreach \y in {0,1} {
      \foreach \z in {0,1,2} {
        \draw[gray] (\x,\y,\z) -- (\x,\y+1,\z);
      }
    }
  }

  \foreach \x in {0,1,2} {
    \foreach \y in {0,1,2} {
      \foreach \z in {0,1} {
        \draw[gray] (\x,\y,\z) -- (\x,\y,\z+1);
      }
    }
  }

  \draw[blue, thick] (1,1,1) -- (1,0,1) -- (1,0,2) -- (1,1,2) -- cycle;
  \draw[blue, thick] (1,1,1) -- (1,2,1) -- (1,2,2) -- (1,1,2) -- cycle;
  \draw[red, thick] (1,1,1) -- (2,1,1) -- (2,1,2) -- (1,1,2) -- cycle;
  \node[red] at (1.5,1.5,1) {$\gamma_1''$};
  \node[blue] at (0,1,1) {$\gamma_2''$};
\end{tikzpicture}
\end{minipage}
\begin{minipage}{0.3\textwidth}
\centering
\begin{tikzpicture}[tdplot_main_coords, scale=1]


  \foreach \x in {0,1} {
    \foreach \y in {0,1,2} {
      \foreach \z in {0,1,2} {
        \draw[gray] (\x,\y,\z) -- (\x+1,\y,\z);
      }
    }
  }

  \foreach \x in {0,1,2} {
    \foreach \y in {0,1} {
      \foreach \z in {0,1,2} {
        \draw[gray] (\x,\y,\z) -- (\x,\y+1,\z);
      }
    }
  }

  \foreach \x in {0,1,2} {
    \foreach \y in {0,1,2} {
      \foreach \z in {0,1} {
        \draw[gray] (\x,\y,\z) -- (\x,\y,\z+1);
      }
    }
  }

  \draw[blue, thick] (1,1,1) -- (2,1,1) -- (2,1,2) -- (1,1,2) -- cycle;
  \draw[blue, thick] (1,1,1) -- (1,2,1) -- (1,2,2) -- (1,1,2) -- cycle;
  \draw[red, thick] (1,1,1) -- (1,0,1) -- (1,0,2) -- (1,1,2) -- cycle;
  \node[red] at (1.5,0.5,1) {$\gamma_1'''$};
  \node[blue] at (1.5,1.5,1) {$\gamma_2'''$};

\end{tikzpicture}
\end{minipage}
\caption{The loop $\gamma$ from Figure \ref{fig:u10negative} written as sums of two loops in three distinct ways.}
\label{fig:Partitonsu10gamma}
\end{figure}
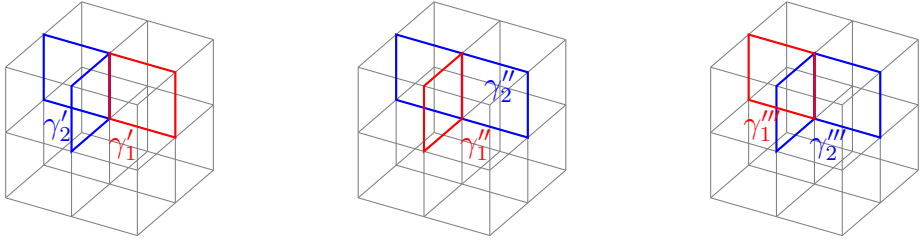
    \newline
    Also, we note that any other partition of $\gamma$ doesn't form closed loops, so that by \eqref{eq:Ursellonedges} and Lemma \ref{lem:Urselledges} we have,
    \begin{equation}
    \label{eq:Ursellinpfthm11} 
    U_{10}^{\sigma}(W_\gamma)=\E_\beta[W_{\gamma}]-\E_\beta[W_{\gamma_1'}]\E_\beta[W_{\gamma_2'}]-\E_\beta[W_{\gamma_1''}]\E_\beta[W_{\gamma_2''}]-\E_\beta[W_{\gamma_1'''}]\E_\beta[W_{\gamma_2'''}].
    \end{equation}
 Now, by Lemma \ref{lem:limitbeta} applied to every expectation in \eqref{eq:Ursellinpfthm11}, if one let $\beta$ increase then every expectation  approaches $1$ in \eqref{eq:Ursellinpfthm11} which leads to
 \begin{equation}
 \label{eq:U10negative}
 U_{10}(\sigma(e_1),...,\sigma(e_{10})) < 0.
 \end{equation}
This proves the claim in the case $n=10$ and $m \geq 3$. \newline 
    Now assume that $n=12$. In this case, one can take the loop $\gamma$ from Figure \ref{fig:u10negative} and add some edges to make one of the loops in the decomposition of $\gamma$ larger, as shown, for example, in Figure \ref{fig:u12neg}. In that case, the new loop called $\tilde{\gamma}$ composed of 12 edges can also be divided into a sum of two loops in exactly three distinct ways: one as a sum of two loops of size $6$ and $6$, and two sums with loops of size $8$ and $4$. The same argumentation as for $\gamma$ above applies to obtain that for large values $\beta >0$, $U_{12}^{\sigma} (W_\gamma) < 0$.
    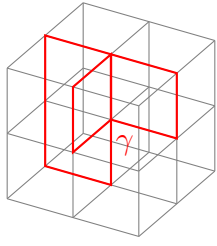
\begin{figure}[ht]
\centering
\begin{tikzpicture}[tdplot_main_coords, scale=1]


  \foreach \x in {0,1} {
    \foreach \y in {0,1,2} {
      \foreach \z in {0,1,2} {
        \draw[gray] (\x,\y,\z) -- (\x+1,\y,\z);
      }
    }
  }

  \foreach \x in {0,1,2} {
    \foreach \y in {0,1} {
      \foreach \z in {0,1,2} {
        \draw[gray] (\x,\y,\z) -- (\x,\y+1,\z);
      }
    }
  }

  \foreach \x in {0,1,2} {
    \foreach \y in {0,1,2} {
      \foreach \z in {0,1} {
        \draw[gray] (\x,\y,\z) -- (\x,\y,\z+1);
      }
    }
  }

  \draw[red, thick] (1,1,1) -- (2,1,1) -- (2,1,2) -- (1,1,2) -- cycle;
  \draw[red, thick] (1,1,1) -- (1,1,0) -- (1,0,0) -- (1,0,1) --(1,0,2) -- (1,1,2) -- cycle;
  \draw[red, thick] (1,1,1) -- (1,2,1) -- (1,2,2) -- (1,1,2) -- cycle;
  \node[red] at (1.5,1.5,1) {$\gamma$};

\end{tikzpicture}
\caption{Example of a loop $\tilde{\gamma}$, s.t. $|\tilde{\gamma}|=12$ and $U_{12}^{\sigma}(W_{\tilde{\gamma}}) < 0$ at large $\beta$.}
\label{fig:u12neg}
\end{figure}
    \newline 
For the general case, this argumentation applies to every $n=2k$, when $k \geq 6$, as one can always increase the size of the rectangular loop in $\tilde{\gamma}$ by adding two new edges. This proves that for $m \geq 3$ if $n \geq 10$ is even, then the Ursell function applied to edges can take negative values when $\beta$ is large enough. \newline 
    It now remains to show that the Ursell function applied on edges can't take negative values if $n=4,6,8$. 
    For $n=4,6$, we note that it is not possible to divide a loop of length $n$ into two closed paths, as a closed path must have a size of at least 4 edges. For $n=8$, if a set of $8$ edges can be partitioned into two loops, then each loop has to contain $4$ edges. It follows immediately that the two loops are either disjoint or share a "corner". In either case it implies
    \[ 
    U_8(\sigma(e_1),...,\sigma(e_8))=\E[W_{\gamma_1+\gamma_2}]-\E[W_{\gamma_1}]\E[W_{\gamma_2}] \geq 0,
    \]
    by Griffith's second inequality. This concludes the proof of (d). \newline 
  (e)  
    Now, assume that $m=2$. If $e_1,...,e_n$ are edges which form a loop $\hat{\gamma}$ that can only be decomposed in one specific partition $\hat{\gamma}=\hat{\gamma}_1+\hat{\gamma}_2$, then one obtains due to Griffith's inequality that,
    \[
    U_n^\sigma(W_{\hat{\gamma}}) = \E[W_{\hat{\gamma}}]-\E[W_{\hat{\gamma}_1}]\E[W_{\hat{\gamma}_2}] \geq 0.
    \]
Hence, one searches for a decomposition of $\gamma$ in at least two distinct loops $\gamma_1,\gamma_2 \neq \gamma_1',\gamma_2'$, s.t.,
    \[
    \gamma = \gamma_1+\gamma_2 = \gamma_1'+\gamma_2'.
    \]
    By Lemma \ref{lem:abouttheloopdec}, the above decomposition implies that $|\gamma| \geq 16$.
    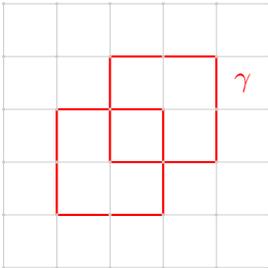
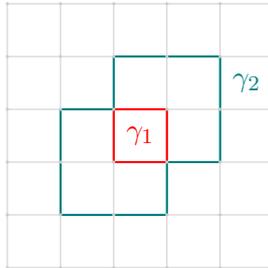
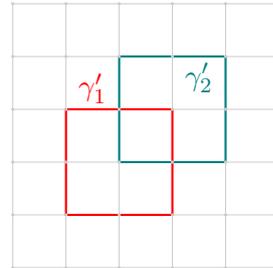
\begin{figure}[ht]
    \centering
    \begin{subfigure}[b]{0.32\textwidth}
\begin{tikzpicture}[scale=0.7, every node/.style={font=\small}]
  \draw[step=1cm, very thin, gray!40] (0,0) grid (5,5);

\draw[thick, red] (2,2) rectangle (3,3);
\draw[thick, red] (2,3) -- (1,3);
\draw[thick, red] (1,3) -- (1,2);
\draw[thick, red] (1,2) -- (1,1);
\draw[thick, red] (1,1) -- (3,1);
\draw[thick, red] (3,1) -- (3,2);
\draw[thick, red] (3,2) -- (4,2);
\draw[thick, red] (4,2) -- (4,4);
\draw[thick, red] (4,4) -- (2,4);
\draw[thick, red] (2,4) -- (2,3);
\node[red] at (4.5,3.5) {$\gamma$};

  \foreach \x in {0,...,5}
    \foreach \y in {0,...,5}
      \fill[gray!40] (\x,\y) circle (0.03);

\end{tikzpicture}
    \caption{Configuration $\gamma$.}
    \label{fig:PfFig5a}
    \end{subfigure}
\hfill
\begin{subfigure}[b]{0.32\textwidth}
    \centering

\begin{tikzpicture}[scale=0.7, every node/.style={font=\small}]
  \draw[step=1cm, very thin, gray!40] (0,0) grid (5,5);

\draw[thick, red] (2,2) rectangle (3,3);
\draw[thick, teal] (2,3) -- (1,3);
\draw[thick, teal] (1,3) -- (1,2);
\draw[thick, teal] (1,2) -- (1,1);
\draw[thick, teal] (1,1) -- (3,1);
\draw[thick, teal] (3,1) -- (3,2);
\draw[thick, teal] (3,2) -- (4,2);
\draw[thick, teal] (4,2) -- (4,4);
\draw[thick, teal] (4,4) -- (2,4);
\draw[thick, teal] (2,4) -- (2,3);
\node[red] at (2.5,2.5) {$\gamma_1$};
\node[teal] at (4.5,3.5) {$\gamma_2$};

  \foreach \x in {0,...,5}
    \foreach \y in {0,...,5}
      \fill[gray!40] (\x,\y) circle (0.03);

\end{tikzpicture}

    \caption{Partition $\gamma=\gamma_1+\gamma_2$.}
    \label{fig:PfFig5b}
\end{subfigure}
    \begin{subfigure}[b]{0.32\textwidth}
\begin{tikzpicture}[scale=0.7, every node/.style={font=\small}]
  \draw[step=1cm, very thin, gray!40] (0,0) grid (5,5);

\draw[thick, red] (1,1) rectangle (3,3);


\node[red] at (1.5,3.4) {$\gamma_1'$};
\draw[thick, teal] (2,2) rectangle (4,4);
\node[teal] at (3.5,3.6) {$\gamma_2'$};

  \foreach \x in {0,...,5}
    \foreach \y in {0,...,5}
      \fill[gray!40] (\x,\y) circle (0.03);

\end{tikzpicture}
    \caption{Partition $\gamma=\gamma_1'+\gamma_2'$.}
    \label{fig:PfFig5c}
    \end{subfigure}
\caption{Configuration $\gamma$ with $|\gamma|=16$ such that $U_{16}^{\sigma}(W_\gamma) < 0$ at large $\beta$.}
\label{fig:PfFig5}
\end{figure}
\newline 
    If $n=16$, observe the configuration $\gamma$ in Figure \ref{fig:PfFig5}. It has two distinct partitions, and hence, one obtains
    \[
    U_{16}(\sigma(e_1),...,\sigma(e_{16}))=\E_\beta[W_{\gamma}]-\E_\beta[W_{\gamma_1}]\E_\beta[W_{\gamma_2}]-\E_\beta[W_{\gamma_1'}]\E_\beta[W_{\gamma_2'}],
    \]
    which can be made negative for large values of $\beta$
    using  \eqref{eq:limto1whenm2} in the case $\beta \to \infty$. \newline  
    Now, for any $n$ even such that $n \geq 18$, one can have a similar configuration as in Figure \ref{fig:PfFig5} and make, e.g., $\gamma'_1$ larger by adding an even number of edges. This concludes the proof.
\end{proof}
\section{Cluster Expansion on Ursell Functions}
\label{sec:Ursellfct}
We will now focus on the proof of Theorem \ref{thm:1}. In this section, we will apply cluster expansion, as described in Section \ref{subsec:ClusterExpansion}, to the summands of Ursell functions. For this reason, we assume that $\beta \ge \beta_0$ is large enough such that cluster expansion is well-defined, i.e.,  the series in \eqref{eq:expwilsonloop} converges absolutely. From now on, we assume that the dimension of the lattice is any fixed natural number $m \geq 3$ and  $n \in \N$ is the parameter of the Ursell function $U_n$. 
Furthermore, we let $\gamma_1,...,\gamma_n$ be any Wilson loops. We will be exactly specifying the shapes of these loops in Section \ref{sec:Proofmainthm}.
\subsection{Cluster expansion for sums of loops}
\label{subsec:UrsellfctClusterexpintro}
In this section, we will describe how to represent the following term,
\[
\E[\prod_{i \in P} W_{\gamma_i}], ~~ P \subset [n],
\]
using a cluster expansion. First, we denote by $\gamma$ the formal sum of the loops $\gamma_1,...,\gamma_n$, namely
\[
\gamma :=\sum_{i \in [n]} \gamma_i.
\]
For each $i \in [n]$ we denote by $q_i \in C_2(B_N,\Z)$ an oriented surface such that $\partial q_i = \gamma_i$. \newline 
For $P \subset [n]$, we define
\begin{equation}
\label{eq:productofW}
    W_{\gamma_P} := \prod_{i \in P} W_{\gamma_i} ~~ \text{and} ~ W_{\gamma} := \prod_{i \in [n]} W_{\gamma_i}.
\end{equation}
In particular, it implies that 
\begin{equation}
\label{eq:oddnumberofW}
   \{(\gamma_i)_{i \in P} : W_{\gamma_P}  = - 1 \} = \{ (\gamma_i)_{i \in P} : W_{\gamma_i} = -1 ~ \text{for an odd number of} ~ i \in P\}.
\end{equation}
Similarly, one defines the sum of the oriented surfaces $(q_i)_{i \in P}$, $P \subset [n]$, by
\[
q_P := \sum_{i \in P} q_i,
\]
which implies that $q_P \in C_2(B_N,\Z)$. Note that $q_P$ is not an oriented surface by our definition, as its boundary is not a closed loop. By \eqref{eq:sumtwochains}, one has
\begin{equation}
\label{eq:againanewequationn}
    \nu(q_P) = \sum_{p \in q_P} \nu(p) = \sum_{i \in P} \sum_{p \in q_i} \nu(p) = \sum_{i \in P} \nu(q_i).
\end{equation}
In particular, this implies that $\rho(\nu(q_P)) = \prod_{i \in P} \rho(\nu(q_i))$. Here, similarly as in Section \ref{subsec:ClusterExpansion}, we say that $\mathcal{V}$ interacts with $q_P$ if $\mathcal{V}(q_P)=1$. 
From \eqref{eq:againanewequationn}, one deduces that
\begin{equation}
    \label{eq:againanewequationn1}
    \mathcal{V}(q_P)=\sum_{\nu \in \mathcal{V}} \nu(q_P) = \sum_{\nu \in \mathcal{V}} \sum_{i \in P} \nu(q_i) = \sum_{i \in P} \mathcal{V}(q_i).
\end{equation}
We now prove a similar result to \eqref{eq:expwilsonloop} for the product of the Wilson observables. 
\begin{Lemma}
\label{lem:thefirstonein4}
    Consider Ising lattice gauge theory on $\Z^m$ with parameter \mbox{$\beta \geq \beta_0$}. Then for any  loops $(\gamma_i)_{i \in P}$, $P \subset [n]$, we have
    \[
    -\mylog(\E[W_{\gamma_P}])= \sum_{\mathcal{V} \in \Xi} 2\Psi_{\beta} (\mathcal{V})\mathcal{V}(q_P).
    \]
\end{Lemma}
\begin{proof}
Assume that $B_N$ is large enough such that it contains all loops $(\gamma_i)_{i \in P}$ and their respective oriented surfaces. By definition and equation \eqref{eq:oddnumberofW} one has,
    \[
    \E_{N,\beta}[W_{\gamma_P}]= \sum_{\sigma \in \Omega_1(B_N,\Z_2)} \prod_{i \in P} \rho(\sigma(\gamma_i)) \mu_{N,\beta}(\sigma). 
    \]
    Let $\sigma \in \Omega_1(B_N,\Z_2)$ and let $\nu = d\sigma$, then by Stokes' theorem \eqref{eq:Stokesthm}, one has the following for all $i \in P$
    \[
    \rho(\sigma(\gamma_i))= \rho(\nu(q_i)).
    \]
    Using \eqref{eq:againanewequationn} one deduces the following
    \[
    \E[W_{\gamma_P}] = \sum_{\nu \in \Omega_2^0(B_N,\Z_2)} \mu_{N,\beta} (\{\sigma : d\sigma= \nu\})\rho(\nu(q_P)).
    \]
    We apply \eqref{eq:newpartfct1} to this equation, and obtain
    \begin{equation}
    \label{lem:eqinlemma41thesecondone}
    \E[W_{\gamma_P}] = \frac{1}{\tilde{Z}_{\beta,N}} \sum_{\nu \in \Omega_2^0(B_N,\Z_2)} \phi_\beta(\nu)\rho(\nu(q_P)).
    \end{equation}
    Now, from \cite[Lemma 3.3]{firstpaper}, which establish \eqref{eq:partfct3}, it follows that
    \[
    \mylog(\tilde{Z}_{\beta,N}) = \mylog \sum_{\nu \in \Omega_2^0(B_N,\Z_2)} \phi_\beta(\nu) = \sum_{\mathcal{V}\in \Xi} \Psi_\beta(\mathcal{V}).
    \]
     Now, if  one replaces the weight $\phi_\beta(\nu)$ by $\phi_\beta(\nu)\rho(\nu(q_P))$ in the proof of \cite[Lemma~3.3]{firstpaper}, then one obtains,
\begin{equation}
\label{eq:eqinlemma41thesecondone2}
    \sum_{\mathcal{V} \in \Xi} \Psi(\mathcal{V})\rho(\mathcal{V}(q_P)) = \mylog \sum_{\nu \in \Omega^0_2(B_N,\Z_2)}  \phi_\beta(\nu)\rho(\nu(q_P)).
\end{equation}
Combining the above, one obtains
\begin{align*}
\mylog(\E[W_{\gamma_P}]) &= \mylog(\tilde{Z}_{\beta,N}^{-1}) + \mylog \sum_{\nu \in \Omega^0_2(B_N,\Z_2)}  \phi_\beta(\nu)\rho(\nu(q_P)) \\  =& - \sum_{\mathcal{V} \in \Xi} \Psi_{\beta} (\mathcal{V})(1-\rho\big(\mathcal{V}(q_P))\big).
\end{align*}
Finally, the claim follows directly from the last equality in \eqref{eq:expwilsonloop}, as the gauge group is $\Z_2$.
\end{proof}
We can now derive an initial expression for the Ursell function using cluster expansion.
\begin{Lemma}
\label{lem:neverending}
    Consider Ising lattice gauge theory at inverse temperature $\beta \geq \beta_0$. Then, for any Wilson loops $(\gamma_i)_{i \in [n]}$, the Ursell function of the Wilson observables equals
    \begin{align}
    \label{eq:defUn}
    U_n^W(W_\gamma) :=& ~ U_n(W_{\gamma_1},...,W_{\gamma_n}) \nonumber \\  
    = &\sum_{\mathfrak{P} \in \mathfrak{F}_n} (-1)^{|\mathfrak{P}|-1} (|\mathfrak{P}|-1)! \myexp \Big( - \sum_{\mathcal{V} \in \Xi} 2\Psi_{\beta}(\mathcal{V}) \sum_{P \in \mathfrak{P}}  \mathcal{V}(q_P) \Big).
    \end{align}
\end{Lemma}
For simplicity we will denote the summands of $U_n^W(W_\gamma)$ by $a_{\mathfrak{P}}$ for $\mathfrak{P} \in \mathfrak{F}_n$, i.e., let
\begin{equation}
    \label{eq:apfromremark}
    a_{\mathfrak{P}} := (-1)^{|\mathfrak{P}|-1} (|\mathfrak{P}|-1)! \myexp \Big( - \sum_{\mathcal{V} \in \Xi} 2\Psi_{\beta}(\mathcal{V}) \sum_{P \in \mathfrak{P}}  \mathcal{V}(q_P) \Big).
\end{equation}
\begin{proof}[Proof of Lemma \ref{lem:neverending}]
    We use the notation from \eqref{eq:productofW} in the definition of the Ursell function in \eqref{eq:Wilson} to obtain
\begin{equation*}
\label{eq:hopefullyoneofothelast}
  U_n^W(W_\gamma)  = \sum_{\mathfrak{P}\in \mathfrak{F}_n} (-1)^{|\mathfrak{P}|-1} (|\mathfrak{P}|-1)! \prod_{P \in \mathfrak{P}} \mathbb{E}\left[W_{\gamma_P}\right].
\end{equation*}
By Lemma \ref{lem:thefirstonein4}, one has
\begin{equation*}
\label{eq:hopefullyoneofothelast2}
-\mylog \prod_{P \in \mathfrak{P}} \mathbb{E}\left[W_{\gamma_P}\right] =  \sum_{P \in \mathfrak{P}} -\mylog\E[W_{\gamma_P}] = \sum_{P \in \mathfrak{P}}\sum_{\mathcal{V} \in \Xi} 2\Psi_\beta(\mathcal{V})\mathcal{V}(q_P).
\end{equation*}
Combining the two above equations gives
\begin{align*}
\label{eq:defUn}
&U_n^W(W_\gamma) = \sum_{\mathfrak{P}\in \mathfrak{F}_n} (-1)^{|\mathfrak{P}|-1} (|\mathfrak{P}|-1)! \myexp \Big(\mylog \prod_{P \in \mathfrak{P}} \mathbb{E}\left[W_{\gamma_P}\right] \Big) =\sum_{\mathfrak{P} \in \mathfrak{F}_n} a_{\mathfrak{P}},
\end{align*}
where $a_\mathfrak{P}$ is represented in \eqref{eq:apfromremark}. This concludes the proof.
\end{proof}
\subsection{Factorization of the Ursell function}
In this section, we will see a way to factorize and represent $U_n^W(W_\gamma)$ in a convenient way for the proof of Theorem \ref{thm:1} in the next section. Notably, for each vortex cluster $\mathcal{V} \in \Xi$, we want to make it clear with which oriented surfaces $(q_i)_{i \in [n]}$ it interacts. For this reason, we introduce the following subsets of $\Xi$. For $I \subset [n]$ define
\begin{equation}
\label{eq:XiI}
\Xi[I] := \{\mathcal{V} \in \Xi : \mathcal{V}(q_i)=  1 \iff i \in I\} \subset \Xi.
\end{equation}
In other words, $\Xi[I]$ is the set of vortex clusters which interact with the oriented surfaces $(q_i)_{i \in I}$. 
\begin{Remark}
\label{rem:XI}
Any $\mathcal{V} \in \Xi$ is in a unique $\Xi[I]$ (where $\mathcal{V} \in \Xi[\emptyset]$ if $\mathcal{V}(q_i)=0$ for all $i \in [n]$). Hence $(\Xi[I])_{I \subset [n]}$ is a partition of $\Xi$.
\end{Remark}
For $I \subset [n]$, define
\begin{equation}
    \label{eq:notation}
		\Psi_\beta[I]:= \sum_{\mathcal{V} \in \Xi[I]} \Psi_\beta (\mathcal{V}).
\end{equation}
Then, for any subset $I \subset [n]$, $\Psi_\beta[I]$ represents the sum of correlated activities of those clusters that interact exactly with the oriented surfaces $q_i$ for $i \in I$. \newline 
For any $I \subset [n]$ and $\mathfrak{P} \in \mathfrak{F}_n$, we also define     
\begin{equation}
        \label{eq:defCIp}
        c_{I,\mathfrak{P}} := \sum_{P \in \mathfrak{P}} \mathbbm{1}\{|P \cap I|~\text{is odd}\}.  
\end{equation}
Before proceeding, let us recall some basic results for $c_{I,\mathfrak{P}}$, which will be useful for understanding later derivations.
\begin{Lemma} 
\label{lem:comesatend}
Consider the term $c_{I,\mathfrak{P}}$ from \eqref{eq:defCIp}, where $I \subset [n]$ and $\mathfrak{P} \in \mathfrak{F}_n$.
Then the following holds. \newline 
(a) For any $\mathfrak{P} \in \mathfrak{F}_n$, we have $c_{\emptyset,\mathfrak{P}}=0$. \newline 
(b) If $\mathfrak{P}=\{[n]\}$ and $I \subset [n]$, we have $
c_{I,\mathfrak{P}} = \mathbbm{1}\{|I| ~ \text{odd}\}$. \newline 
(c) For any $\mathfrak{P} \in \mathfrak{F}_n$ and any $i \in [n]$, we have $c_{\{i\},\mathfrak{P}}=1$.
\newline 
(d) For any $\mathfrak{P} \in \mathfrak{F}_n$ and any $I \subset [n]$, we have $
c_{I,\mathfrak{P}} \in \{0,...,n\}$. \newline 
(e) For any $I=\{i,j\} \subset [n], i \neq j$, then we have 
    \begin{equation*}      
        c_{I, \mathfrak{P}} = 2 \cdot \mathbbm{1}\{i,j ~ \text{are in different partitions elements of } \mathfrak{P}\}.
    \end{equation*}
    In particular if $|I|=2$, then $c_{I,\mathfrak{P}} \neq 0$ if and only if $c_{I,\mathfrak{P}} = 2$. \newline 
(f) Let $\mathfrak{P} \in \mathfrak{F}_n \setminus \{[n]\}$. Then there exists $i \in [n-1]$, such that for $I=\{i,i+1\}$ we have $c_{I,\mathfrak{P}} \neq 0$.
\end{Lemma}
\begin{proof}
    (a), (b), (c), (d), and (e) follow immediately from \eqref{eq:defCIp}. \newline 
    (f)  As $\mathfrak{P} \neq \{[n]\}$, $\mathfrak{P}$ consists of at least two partition elements. 
    Choose $P_1 \in \mathfrak{P}$ such that $\min P_1 \neq 1$. We denote it by $i+1 :=\min ~ P_1$. Then $i = \min P_1 -1$ must be in some other set $P_2 \in \mathfrak{P}$. 
    In particular, by (e) it holds that $c_{\{i,i+1\},\mathfrak{P}} = 2$. This concludes the proof.
\end{proof}
We now derive a representation of $U_n^W(W_\gamma)$ using the above-introduced notation.
\begin{Lemma}
    \label{lem:Leamnewap}
    Consider Ising lattice gauge theory with parameter $\beta \geq \beta_0$, and $(\gamma_i)_{i \in [n]}$ being $n$ Wilson loops. Then, 
        \begin{equation}
        \label{eq:newap}
        U_n^W(W_\gamma) = \sum_{\mathfrak{P} \in \mathfrak{F}_n} (-1)^{|\mathfrak{P}|-1}(|\mathfrak{P}|-1)! \myexp \Big(- 2 \sum_{I \subseteq [n]} c_{I,\mathfrak{P}} \Psi_\beta[I] \Big).
        \end{equation}
    \end{Lemma}
\begin{proof}
Using the expression of $U_n$ from Lemma \ref{lem:neverending}, the desired conclusion will follow if the following equality holds for all $\mathfrak{P} \in \mathfrak{F}_n$
\begin{equation}
\label{eq:Lem45toprove}
 2\sum_{\mathcal{V} \in \Xi} \Psi_{\beta}(\mathcal{V}) \sum_{P \in \mathfrak{P}} \mathcal{V}(q_P) =  2 \sum_{I \subseteq [n]} c_{I,\mathfrak{P}} \Psi_\beta[I].
\end{equation}
To this end let $\mathfrak{P} \in \mathfrak{F}_n$.
Note that if $\mathcal{V} \in \Xi[I]$ for some $I \subset [n]$, then
\[
\mathcal{V}(q_i)= \mathbbm{1}\{i \in I \}.
\]
Which implies
\begin{equation}
\label{eq:cooltohavethisone}
    \sum_{P \in \mathfrak{P}} \mathcal{V}(q_P) = \sum_{P \in \mathfrak{P}} \sum_{i \in P} \mathcal{V}(q_i)= \sum_{P \in \mathfrak{P}} \mathbbm{1}\{|P \cap I|~\text{is odd}\}=c_{I,\mathfrak{P}}.
\end{equation}
    Hence
\begin{align*}
\sum_{\mathcal{V}\in\Xi} \Psi_\beta(\mathcal{V}) 
  \sum_{P\in\mathfrak{P}} \mathcal{V}(q_P)
&= \sum_{I\subset[n]} 
     \sum_{\mathcal{V}\in\Xi[I]} 
       \Psi_\beta(\mathcal{V})
       \sum_{P\in\mathfrak{P}} \mathcal{V}(q_P) \\[0.3em]
&= \sum_{I\subset[n]} 
     c_{I,\mathfrak{P}}
     \sum_{\mathcal{V}\in\Xi[I]} \Psi_\beta(\mathcal{V}) \\[0.3em]
&= \sum_{I\subset[n]} 
     c_{I,\mathfrak{P}}\,
     \Psi_\beta[I].
\end{align*}
The first equality comes from the fact that $(\Xi[I])_{I \subset [n]}$ is a partition of $\Xi$, the second from \eqref{eq:cooltohavethisone}, and the last from the definition of $\Psi_\beta[I]$. \newline 
This concludes the proof.
\end{proof}
From Lemma \ref{lem:comesatend}(b)  and Lemma \ref{lem:Leamnewap}, one deduces that 
\begin{equation}
    \label{eq:newsummandfirst}
    a_{\{[n]\}}=\myexp\Big(-2 \sum_{I \subset [n]}\mathbbm{1}\{|I| ~ \text{odd}\} \Psi_\beta[I]\Big).
\end{equation}
We now have the ingredients to factorize $U_n$.
    \begin{Lemma}
    \label{lem:factorizationUn}
Consider Ising lattice gauge theory with parameter $\beta \geq \beta_0$. Then one can factorize $U_n$ as
\begin{equation}
\label{eq:oneofUnsfactorizations}
        U_n^W(W_\gamma) = a_{\{[n]\}} (1 + V_n(W_\gamma)),
        \end{equation}
        where $V_n(W_\gamma)$ is defined by
        \begin{equation}
        \label{eq:defVn}
        V_n(W_\gamma) := \sum_{\substack{\mathfrak{P} \in \mathfrak{F}_n : \\ \mathfrak{P} \neq \{[n]\}}} b_\mathfrak{P}.
        \end{equation}
        Here the coefficients $b_\mathfrak{P}$, for $\mathfrak{P} \in \mathfrak{F}_n, ~ \mathfrak{P} \neq \{[n]\}$, are defined by
        \begin{equation}
\label{eq:defbp}
b_{\mathfrak{P}} := \frac{a_{\mathfrak{P}}}{a_{\{[n]\}}} = (-1)^{|\mathfrak{P}|-1} (|\mathfrak{P}|-1)! \frac{\myexp \Big(- 2 \sum_{I \subseteq [n]} c_{I,\mathfrak{P}} \Psi_\beta[I]\Big)}{\myexp\big( -2 \sum_{I \subset [n]}\mathbbm{1}\{|I| ~ \text{odd}\} \Psi_\beta[I]\Big)}.
\end{equation}
Moreover, the coefficients $(b_{\mathfrak{P}})_{\mathfrak{P} \in \mathfrak{F}_n \setminus \{[n]\}}$ do not depend on the weights of the vortex clusters $\mathcal{V} \in \Xi[\{i\}]$ for any $i \in [n]$.
    \end{Lemma}
\begin{proof}
    By definition of $V_n(W_\gamma)$ and its coefficients $(b_{\mathfrak{P}})_{\mathfrak{P} \in \mathfrak{F}_{n \setminus \{[n]\}}}$, it is clear that \eqref{eq:oneofUnsfactorizations} holds. \newline 
    Now, let $\mathfrak{P} \in \mathfrak{F}_{n}$ and $i \in [n]$. By Lemma \ref{lem:Leamnewap}, we have
    \[
    a_\mathfrak{P} = (-1)^{|\mathfrak{P}|-1}(|\mathfrak{P}|-1)! \myexp \Big(- 2 \Big(\sum_{\substack{I \subseteq [n] : \\ |I| \ge 2}} c_{I,\mathfrak{P}} \Psi_\beta[I] + \sum_{i \in [n]} c_{\{i\},\mathfrak{P}}\Psi_\beta[\{i\}]\Big) \Big).
    \]
    Now, by Lemma \ref{lem:comesatend}(c) the coefficient $c_{\{i\}, \mathfrak{P}}$ equals $1$ and does not depend on the choice of $\mathfrak{P} \in \mathfrak{F}_n$. 
    This implies that $b_\mathfrak{P}$ does not depend on $\Psi_\beta[\{i\}]$, as the term gets canceled in the division. That concludes the proof.
\end{proof}
One can now take advantage of the factorization to derive the sign of $U_n$. 
Using the definition of $V_n(W_\gamma)$, depending on whether $|\mathfrak{P}|$ is odd or even, one can divide $V_n(W_\gamma)$ into a positive and a negative part. Simply note that for $\mathfrak{P} \in \mathfrak{F}_n \setminus \{[n]\}$
\[
\begin{cases}
    |\mathfrak{P}|~ \text{odd implies} ~ b_{\mathfrak{P}} \ge 0, \\
    |\mathfrak{P}|~ \text{even implies} ~ b_{\mathfrak{P}} \le 0.
\end{cases}
\]
One deduces 
    \begin{equation}
    \label{eq:Vnpositivenegative}
    V_n(W_\gamma)= \sum_{\substack{\mathfrak{P} \in \mathfrak{F}_n \setminus \{[n]\}\\ |\mathfrak{P}| ~ \text{odd}}} b_\mathfrak{P} + \sum_{\substack{\mathfrak{P} \in \mathfrak{F}_n \\ |\mathfrak{P}| ~ \text{even}}} b_\mathfrak{P} =: V^+_n(W_\gamma) + V^-_n(W_\gamma), 
    \end{equation} 
    where $V_n^{+}(W_\gamma) \geq 0$ consists of positive summands  while  $V_n^{-}(W_\gamma) \leq 0$ consists of negative summands. By Lemma \ref{lem:factorizationUn} it follows clearly that,
    \begin{equation}
    \label{eq:decompv}
        U_{n}^W(W_\gamma)=a_{\{[n]\}}\big(1+V_n^+(W_\gamma)+V_n^-(W_\gamma)\big).
    \end{equation}

Equation \eqref{eq:decompv} will prove to be central for understanding the sign of $U_n$ in the last section. \newline 
Lemma \ref{lem:factorizationUn} gives a convenient expression in the next corollary for $V_n^-(W_\gamma)$ that will be used in the final proof.
\begin{Corollary}
\label{cor:firstcorollary}
    Under the same assumptions as in Lemma \ref{lem:factorizationUn}, it holds that
    \begin{equation*}
    V_n^-(W_\gamma)= \sum_{\substack{\mathfrak{P} \in \mathfrak{F}_n  \\ |\mathfrak{P}| ~ \text{even}}} b_\mathfrak{P} = - \sum_{\substack{\mathfrak{P} \in \mathfrak{F}_n  \\ |\mathfrak{P}| ~ \text{even}}} (\mathfrak{P}-1)! \myexp\Big(-2 \sum_{I \subset [n]} (c_{I,\mathfrak{P}} - \mathbbm{1}\{|I| ~ \text{odd}\}) \Psi_\beta[I] \Big).
    \end{equation*}
\end{Corollary}
\begin{proof}
The equality follows directly from Lemma \ref{lem:factorizationUn}.
\end{proof}
\section{Proof of Theorem \ref{thm:1}}
\label{sec:Proofmainthm}
In this section, we will describe a configuration of Wilson loops $(\gamma_i)_{i \in [n]}$ and use those to prove Theorem \ref{thm:1}, that is, at a temperature low enough, the Ursell function of the Wilson loops observable $U_n^W(W_\gamma)$ is positive. 
Our goal is first to obtain leading-order estimates on cluster activities $\Psi_\beta[I]$ for $I \subset [n]$. This will be done in the next three subsections. The last subsection will then collect all previous results to prove Theorem \ref{thm:1}. 
\subsection{Smallest vortex clusters}
\label{subsec:Clusters}
In order to obtain information on the first-order terms of $\Psi_\beta(\mathcal{V})$, $\mathcal{V} \in \Xi$, we first find the vortices $\nu \in \Lambda(B_N)$ with the smallest possible support. 
\begin{Lemma}
\label{lem:sizesupportvortex}
    Let $\nu \in \Lambda(B_N)$ be a vortex in $\Z^m, m \geq 3$ such that the support of $\nu$ doesn't contain any boundary plaquettes of $B_N$. Then, either
    \[
    |\supp~ \nu^+| \in \{2(m-1),4(m-1)-2\} ~~ \text{or} ~~ |\supp ~ \nu^+| \geq 4(m-1).
    \]
\end{Lemma}
\begin{proof}
  The first claim, specifically that~$|\supp~ \nu^+|\in\{2(m-1),4(m-1)-2\}$~is a consequence of \cite[Lemma 2.3]{firstpaper}. 
  Let now $\nu \in \Lambda(B_N)$. We will prove that $\supp \nu^+$ is even, which will prove the lemma.
    Hence, recall that by Poincar\'e's Lemma there exists a one-form $\sigma \in \Omega_1(B_N,\Z_2)$ such that $d\sigma=\nu$. In particular, a positively oriented plaquette $p \in C_2(B_N)^+$ is in the support of $\nu^+$ if and only if there is an odd number of edges $e \in C_1(B_N)^+$ in the boundary of $p$ such that $e \in \supp \sigma^+$. In other words, $p \in \supp \nu^+ \iff |e \in \supp \sigma^+: e \in \partial p| \in \{1,3\}$. Note that $e \in \partial p$ is equivalent to $p \in \hat{\partial} e$. Hence, we deduce the following: every plaquette containing one boundary edge $e \in \supp \sigma^+$ is a candidate to be in the support of $\nu^+$. We then need to subtract all plaquettes that are counted twice or more, i.e., which satisfy $|\{e \in \supp \sigma : p \in \hat{\partial} e\}| \in \{2,3,4\}$. Hence, one obtains
    \begin{align}
    \label{eq:finallysomeestimatesonsupp}
    |\supp ~ \nu^+|=& \sum_{e \in \supp \sigma^+} |\hat{\partial}e| - 2 \Big(\sum_{p \in C_2(B_N)^+}\mathbbm{1}\{|\{e \in \supp ~ \sigma^+ : p \in \hat{\partial}e\}| \geq 2\} \\  & \nonumber + 2\mathbbm{1}\{|\{e \in \supp ~ \sigma^+ : p \in \hat{\partial}e\}| = 4\}\Big).
    \end{align}
    Now, note that every edge $e$ has $2(m-1)$ adjacent plaquettes in $\Z^m$, i.e.,
    \[
    |\hat{\partial}e|=2(m-1).
    \] 
 Thus, by \eqref{eq:finallysomeestimatesonsupp}, $|\supp \nu^+|$ is an even number. Therefore, the next vortex cluster's size is $|\supp \nu^+| = 4(m-1)$. This completes the proof.
\end{proof}
Figure \ref{fig:THEBIGFIGURE} shows the structure of the smallest vortices.
\begin{figure}[ht]
  \centering
  \setlength{\tabcolsep}{8pt}
  \renewcommand{\arraystretch}{1.4}
 \resizebox{0.85\textwidth}{!}{
  \begin{tabular}{c c c c}
    \toprule
    $|\text{supp} ~\nu^+|$ &
    $\text{Support of} ~ \sigma$ s.t. $d\sigma=\nu$ &
    $\text{Support of} ~\nu$ &
    Induced subgraph of $\mathcal{G}_2$ \\
    \midrule

    $2(m-1)$ &
    \SigmaA &
    \NuA &
    \GraphA \\[2ex]

    $4(m-1)-2$ &
    \SigmaB &
    \NuB &
    \GraphB \\[2ex]

    $4(m-1)-2$ &
    \SigmaC &
    \NuC &
    \GraphC \\[2ex]

    $4(m-1)$ &
    \SigmaD &
    \NuD &
    \GraphD \\

    \bottomrule
  \end{tabular}
  }
  \caption{Minimal connected vortices $\nu$ in $\mathbb{Z}^m$ for $m=3$. For each support size, one can see an associated one-form $\sigma$, with $d\sigma=\nu$, the support of the vortex $\nu$ itself, and the induced subgraph of the plaquette adjacency graph $\mathcal{G}_2$.}
  \label{fig:THEBIGFIGURE}
\end{figure}
\subsection{Construction of the loops}
\label{subsec:loopsconstruction}
In this section, we will define and fix a configuration of loops $\gamma_1,..., \gamma_n$. The idea is that the loops are assumed to be \textit{stacked} on top of each other. Now, recall the notation from Section \ref{subsec:DEC}, in which $\tilde{e}_1,\tilde{e}_2,\tilde{e}_3$ are the first three basis vectors from $\Z^m$ and $d\tilde{e}_1,d\tilde{e}_2,d\tilde{e}_3$ are the corresponding positively oriented edges. Recall also that if $\chi \in C_k(B_N,\Z)$ is a k-chain, and $y \in C_k(B_N)^+$ a k-cell, then  $\chi[y]$ is the coefficient associated to the k-cell $y$ in $\chi$.
With these notations in mind, we can formally define the loop and oriented surface configuration as follows.
\begin{Def}
\label{Def:Wilsonloops}
Assume that $B_N$ is large enough, such that all oriented surfaces and loops in this definition are well-defined.
Let $x \in \Z^m$ and $L_1,L_2 \in \N$. Define the discrete rectangle
\[
R := \{x +a \tilde{e}_1+b \tilde{e}_2: 0 \leq a \leq L_1, 0 \leq b \leq L_2\}.
\]
Define the 2-chain $q_1 \in C_2(B_N,\Z)$ as the formal sum of all positively oriented plaquettes $p \in C_2(B_N)^+$ whose four boundary edges lie in the rectangle $R$. \newline 
Then, define the translation operator $\tau$ by
\[
\tau : C_2(B_N,\Z) \to C_2(B_N,\Z), ~ (\tau\chi)[p] := \chi[p-d\tilde{e}_3].
\]
Now we define recursively the two chains $q_i$ for $i \in \{2,...,n\}$, by
\[
q_i := \tau q_{i-1}.
\]
In words, every $q_i$ is the translation of the positively oriented plaquettes of $q_{i-1}$ by $d\tilde{e}_3$. \newline  
Finally, for $i \in [n]$, define
\[
\gamma_i := \partial q_i.
\]
Note that by construction $\partial \gamma_i =0$ so that $(\gamma_i)_{i \in [n]}$ are well defined Wilson loops, and $(q_i)_{i \in [n]}$ well defined oriented surfaces.
\end{Def}
The definition of the oriented surfaces $(q_i)_{i \in [n]}$ as "flat" surfaces will prove to be convenient for future computations, as well as making the oriented surfaces have disjoint support. \newline 
The numbers $L_1$ and $L_2$ will be used solely in the definition. From now on, we will let $|\gamma_1|$ denote the length of the loops, without distinguishing the values of $L_1,L_2$. In the final proof, we will make $|\gamma_1|$ large, without any distinction on the values $L_1,L_2$.
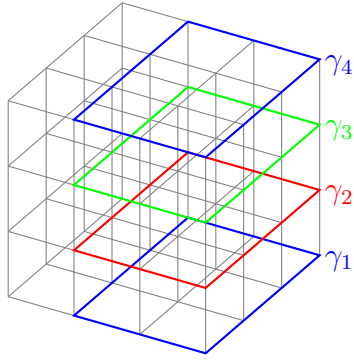
\begin{figure}[ht]
\centering
\begin{tikzpicture}[tdplot_main_coords, scale=1]


  \foreach \x in {0,1,2} {
    \foreach \y in {0,1,2,3} {
      \foreach \z in {0,1,2,3} {
        \draw[gray] (\x,\y,\z) -- (\x+1,\y,\z);
      }
    }
  }

  \foreach \x in {0,1,2,3} {
    \foreach \y in {0,1,2} {
      \foreach \z in {0,1,2,3} {
        \draw[gray] (\x,\y,\z) -- (\x,\y+1,\z);
      }
    }
  }

  \foreach \x in {0,1,2,3} {
    \foreach \y in {0,1,2,3} {
      \foreach \z in {0,1,2} {
        \draw[gray] (\x,\y,\z) -- (\x,\y,\z+1);
      }
    }
  }

   \draw[blue, thick]
  (0,1,0) -- (0,3,0) -- (3,3,0) -- (3,1,0) -- cycle;
\draw[red, thick]
  (0,1,1) -- (0,3,1) -- (3,3,1) -- (3,1,1) -- cycle;
 \draw[green, thick]
  (0,1,2) -- (0,3,2) -- (3,3,2) -- (3,1,2) -- cycle;
  \draw[blue, thick]
  (0,1,3) -- (0,3,3) -- (3,3,3) -- (3,1,3) -- cycle;
    \node[blue] at (0,3.3,0) {$\gamma_1$};
    \node[red] at (0,3.3,1) {$\gamma_2$};
    \node[green] at (0,3.3,2) {$\gamma_3$};
    \node[blue] at (0,3.3,3) {$\gamma_4$};
\end{tikzpicture}
\caption{Example of four loops $\gamma_1,\gamma_2,\gamma_3,\gamma_4$ as defined in Definition \ref{Def:Wilsonloops} for $\Z^3$.}
\label{fig:exfourloops}
\end{figure}
\begin{Remark}
    The above loops were chosen to facilitate computations. With small modifications, one can also prove the main result using more general loops that are not equal up to a translation or are more distant from each other.
\end{Remark}
\newpage
\subsection{Estimates for cluster activities}
Recall the definitions of $\Xi[I] ~ \text{and} ~\Psi_\beta[I]$ from \eqref{eq:XiI} and \eqref{eq:notation}, as well as $V_n^-(W_\gamma)$ from \eqref{eq:Vnpositivenegative}. \newline 
In this section, we will prove that for $\beta$ large enough, the main contributions in the term $V_n^-(W_\gamma)$ come from the correlated activities of clusters that only connect two neighboring loops. These correlated activities will then be used to obtain estimates on $V_n^-(W_\gamma)$, which will lead to the sign of $U_n$ using \eqref{eq:decompv}. The clusters connecting two neighboring loops in our construction correspond to the clusters in $\Xi[\{i,i+1\}]$ for $~ i \in \{1,...,n-1\}$. 
We will first estimate the correlated activity of these terms in~Lemma~\ref{lem:firstorder} as a function of $\beta$ and the length of the loops $|\gamma_1|$. We will then find an upper bound of the correlated activities coming from the other clusters in~Lemma~ \ref{lem:restorders}. \newline 
Recall that by our construction the loops are equal up to a translation, and hence have the same length $|\gamma_1|=...=|\gamma_n|$.
\begin{Lemma}
    \label{lem:firstorder}
    Consider Ising lattice gauge theory with parameter $\beta \ge \beta_0$. Let $(\gamma_i)_{i \in [n]}$ be loops in Definition \ref{Def:Wilsonloops}. Then, for $i \in \{1,...,n-1\}$, it holds that
    \begin{align*}
        &\Psi_\beta[\{i,i+1\}] =  |\gamma_i| \Big( \myexp(-4(4(m-1)-2)\beta) +  \mathcal{O}\big(\myexp{(-4(4(m-1))\beta)}\big)\Big),
    \end{align*}
    where the term $\mathcal{O}\big(\myexp{(-4(4(m-1))\beta)}\big)$ is independent of $|\gamma_i|$.
\end{Lemma}
The idea of the proof is to divide the clusters from $\Xi[\{i,i+1\}]$ into different categories. For this purpose, recall the notation $n(\mathcal{V})$ from \eqref{eq:notation||}. With this notation, we have
\begin{align}
\label{eq:firstorderpsi}
 \Psi_\beta[\{i,i+1\}] = \sum_{\mathcal{V} \in \Xi[\{i,i+1\}]} \Psi_\beta (\mathcal{V}) = \sum_{\substack{\mathcal{V} \in \Xi[\{i,i+1\}] \\ n(\mathcal{V}) = 1}} \Psi_\beta (\mathcal{V})+\sum_{\substack{\mathcal{V} \in \Xi[\{i,i+1\}] \\ n(\mathcal{V}) \geq 2}} \Psi_\beta (\mathcal{V}).
\end{align}
We can now state the following lemma.
\begin{Lemma}
\label{lem:firstordertwosummands}
Under the same assumptions as for Lemma \ref{lem:firstorder}, for $i \in \{1,...,n-1\}$, we have
    \begin{align*}
        \sum_{\substack{\mathcal{V} \in \Xi[\{i,i+1\}] \\ n(\mathcal{V}) = 1}} \Psi_\beta (\mathcal{V}) =   |\gamma_i| \Big(\big( \myexp(-4(4(m-1)-2)\beta) +  \mathcal{O}\big(\myexp{(-4(4(m-1))\beta)}\big)\Big),
    \end{align*}
    where the term $\mathcal{O}\big(\myexp{(-4(4(m-1))\beta)}\big)$ is independent of $|\gamma_i|$.
\end{Lemma}
\begin{Lemma}
    \label{lem:firstordersecondsummand}
Under the same assumptions as for Lemma \ref{lem:firstorder}, for $i \in \{1,...,n-1\}$, we have
    \begin{align*}
         \sum_{\substack{\mathcal{V} \in \Xi[\{i,i+1\}] \\ n(\mathcal{V}) \geq 2}} \Psi_\beta (\mathcal{V})  =  |\gamma_i| \Big(- \myexp{(-4(4(m-1))\beta)} +  o\big(\myexp{(-4(4(m-1))\beta)}\big)\Big), 
    \end{align*}
    where the term $o\big(\myexp{(-4(4(m-1))\beta)}\big)$ is independent of the loop's size.
\end{Lemma}
\begin{proof}[Proof of Lemma \ref{lem:firstorder}]
Combine Lemma \ref{lem:firstordertwosummands} and Lemma \ref{lem:firstordersecondsummand} in equation \eqref{eq:firstorderpsi}. The proof follows immediately.
\end{proof}
In the proofs of Lemma \ref{lem:firstordertwosummands} and Lemma \ref{lem:firstordersecondsummand}, we will make use of the following notation. For $j \in \N$, and $i \in [n-1]$, let
\begin{align*}
&\Xi[\{i,i+1\}]_{j} := \{\mathcal{V} \in \Xi[\{i,i+1\}] : |\mathcal{V}| = j \}, ~ \text{and} \\
& \Xi[\{i,i+1\}]_{j^+} := \{\mathcal{V} \in \Xi[\{i,i+1\}] : |\mathcal{V}| \geq j \},
\end{align*}
where $|\mathcal{V}|$ was introduced in \eqref{eq:notation||}. 
\begin{proof}[Proof of Lemma \ref{lem:firstordertwosummands}]
Let $i \in [n-1]$ and let $\mathcal{V} \in \Xi[\{i,i+1\}]$ be such that $n(\mathcal{V}) = 1$. Then there is a vortex $\nu \in \Lambda(B_N)$ such that $\mathcal{V} = \{\nu\}$ is a singleton. In order to obtain estimates on the correlated activity of $\mathcal{V}$, we will use equation~\eqref{eq:Decompositioncorracti}. As $n(\mathcal{V}) = 1 $ implies that $ \mathcal{U}(\mathcal{V})=1$, we have
\begin{equation}
\label{eq:firstorderestimatepsi}
\Psi_\beta(\mathcal{V})=\myexp(-4\beta|\supp~ \nu^+|).
\end{equation}
Hence, the vortex clusters with the largest contribution to $\Psi_\beta(\mathcal{V})$ are the ones with the smallest support. 
By Lemma \ref{lem:sizesupportvortex}, we know that the support of $\nu^+$ has either size $2(m-1), 4(m-1)-2$ or is larger or equal than $4(m-1)$. Furthermore, as $\mathcal{V} = \{\nu\} \in \Xi[\{i,i+1\}]$, it must hold that $\nu(q_i)=\nu(q_{i+1})=1$, where $q_i$ $q_{i+1}$ are the oriented surfaces specified in Definition \ref{Def:Wilsonloops}. 
Now, with these assumptions, note that it is not possible for $\nu^+$ to have a support of size $2(m-1)$. As $\nu^+$ should connect two disjoint oriented surfaces, a vortex cluster of this size cannot interact with both oriented surfaces, see, e.g.,  Figure \ref{fig:nuofminimsize}.
\begin{figure}[ht]

\centering
\begin{tikzpicture}[tdplot_main_coords, scale=0.8]

  \foreach \x in {0,1,2,3}
    \foreach \y in {0,1,2,3}
      \foreach \z in {0,1,2,3}
        \draw[gray, opacity=0.1] (\x,\y,\z) -- (\x+1,\y,\z);

  \foreach \x in {0,1,2,3}
    \foreach \y in {0,1,2,3}
      \foreach \z in {0,1,2,3}
        \draw[gray, opacity=0.1] (\x,\y,\z) -- (\x,\y+1,\z);

  \foreach \x in {0,1,2,3}
    \foreach \y in {0,1,2,3}
      \foreach \z in {0,1,2,3}
        \draw[gray, opacity=0.1] (\x,\y,\z) -- (\x,\y,\z+1);

  \path[pattern=north east lines, pattern color=red]
    (0,1,0) -- (0,3,0) -- (3,3,0) -- (3,1,0) -- cycle; -- cycle;
  
  \draw[red, thick]
    (0,1,0) -- (0,3,0) -- (3,3,0) -- (3,1,0) -- cycle;

  \path[pattern=north east lines, pattern color=blue]
    (0,1,1) -- (0,3,1) -- (3,3,1) -- (3,1,1) -- cycle;

  \draw[blue, thick]
    (0,1,1) -- (0,3,1) -- (3,3,1) -- (3,1,1) -- cycle;

\draw[black, dashed, thick]
  (0,1,1) -- (0,2,1) -- (0,2,2) -- (0,1,2) -- cycle;

 \draw[black, dashed, thick] (0,1,0) -- (0,2,0) -- (0,2,1) -- (0,1,1) -- cycle;

\draw[black, dashed, thick] (0,2,1) -- (1,2,1) -- (1,1,1) -- (0,1,1) -- cycle;

\draw[black, dashed, thick] (-1,2,1) -- (0,2,1) -- (0,1,1) -- (-1,1,1) -- cycle;

  \node[red] at (-0.2,3.2,-0.3) {$\gamma_1$};
  \node[blue] at (-0.2,3.2,1) {$\gamma_2$};
  \node[black] at (-1,1.2,1.8) {$\nu$};

\end{tikzpicture}

\caption{Two neighboring Wilson loops $\gamma_1$, $\gamma_2$ in $\Z^3$ with dashed oriented surfaces $q_1$ marked in red and $q_2$ in blue, together with a vortex cluster $\nu^+$ with support size $2(m-1)$.}
\label{fig:nuofminimsize}

\end{figure}
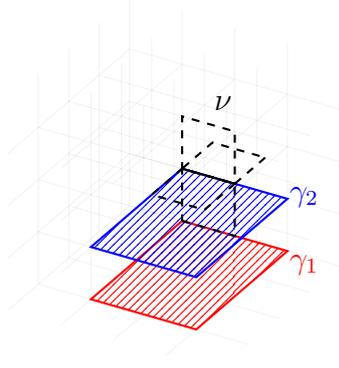
\newline 
If $|\supp ~ \nu^+| = 4(m-1)-2$, then $\nu$ must be one of the two configurations of vortices described in Figure
\ref{fig:nuofsecminimsize}.
\begin{figure}[ht]
\centering
\begin{subfigure}[b]{0.45\textwidth}
\centering
\begin{tikzpicture}[tdplot_main_coords, scale=1]

  \foreach \x in {0,1,2,3}
    \foreach \y in {0,1,2,3}
      \foreach \z in {0,1,2,3}
        \draw[gray, opacity=0.1] (\x,\y,\z) -- (\x+1,\y,\z);

  \foreach \x in {0,1,2,3}
    \foreach \y in {0,1,2,3}
      \foreach \z in {0,1,2,3}
        \draw[gray, opacity=0.1] (\x,\y,\z) -- (\x,\y+1,\z);

  \foreach \x in {0,1,2,3}
    \foreach \y in {0,1,2,3}
      \foreach \z in {0,1,2,3}
        \draw[gray, opacity=0.1] (\x,\y,\z) -- (\x,\y,\z+1);

  \path[pattern=north east lines, pattern color=red]
    (0,1,0) -- (0,3,0) -- (3,3,0) -- (3,1,0) -- cycle; -- cycle;
  
  \draw[red, thick]
    (0,1,0) -- (0,3,0) -- (3,3,0) -- (3,1,0) -- cycle;

  \path[pattern=north east lines, pattern color=blue]
    (0,1,1) -- (0,3,1) -- (3,3,1) -- (3,1,1) -- cycle;

  \draw[blue, thick]
    (0,1,1) -- (0,3,1) -- (3,3,1) -- (3,1,1) -- cycle;

\draw[black, dashed, thick]
  (0,1,1) -- (0,2,1) -- (0,2,2) -- (0,1,2) -- cycle;


\draw[black, dashed, thick] (0,2,1) -- (1,2,1) -- (1,1,1) -- (0,1,1) -- cycle;

\draw[black, dashed, thick] (-1,2,1) -- (0,2,1) -- (0,1,1) -- (-1,1,1) -- cycle;
\draw[black, dashed, thick]
  (0,2,1) -- (-1,2,1) -- (-1,2,0) -- (0,2,0) -- cycle;


\draw[black, dashed, thick] (0,2,1) -- (0,3,1) -- (0,3,0) -- (0,2,0) -- cycle;

\draw[black, dashed, thick] (0,2,1) -- (1,2,1) -- (1,2,0) -- (0,2,0) -- cycle;
  \node[red] at (-0.2,3.2,-0.3) {$\gamma_1$};
  \node[blue] at (-0.2,3.2,1) {$\gamma_2$};
  \node[black] at (-1,1.2,1.8) {$\nu$};

\end{tikzpicture}
\caption{First cluster type $\nu$ s.t. $|\supp~ \nu^+|=4(m-1)-2$.}
\label{fig:loops_a}
\end{subfigure}
\hfill
\begin{subfigure}[b]{0.45\textwidth}
\centering
\begin{tikzpicture}[tdplot_main_coords, scale=1]

  \foreach \x in {0,1,2,3}
    \foreach \y in {0,1,2,3}
      \foreach \z in {0,1,2,3}
        \draw[gray, opacity=0.1] (\x,\y,\z) -- (\x+1,\y,\z);

  \foreach \x in {0,1,2,3}
    \foreach \y in {0,1,2,3}
      \foreach \z in {0,1,2,3}
        \draw[gray, opacity=0.1] (\x,\y,\z) -- (\x,\y+1,\z);

  \foreach \x in {0,1,2,3}
    \foreach \y in {0,1,2,3}
      \foreach \z in {0,1,2,3}
        \draw[gray, opacity=0.1] (\x,\y,\z) -- (\x,\y,\z+1);

  \path[pattern=north east lines, pattern color=red]
    (0,1,0) -- (0,3,0) -- (3,3,0) -- (3,1,0) -- cycle; -- cycle;
  
  \draw[red, thick]
    (0,1,0) -- (0,3,0) -- (3,3,0) -- (3,1,0) -- cycle;

  \path[pattern=north east lines, pattern color=blue]
    (0,1,1) -- (0,3,1) -- (3,3,1) -- (3,1,1) -- cycle;

  \draw[blue, thick]
    (0,1,1) -- (0,3,1) -- (3,3,1) -- (3,1,1) -- cycle;

\draw[black, dashed, thick]
  (0,1,1) -- (0,2,1) -- (0,2,2) -- (0,1,2) -- cycle;


\draw[black, dashed, thick] (0,2,1) -- (1,2,1) -- (1,1,1) -- (0,1,1) -- cycle;

\draw[black, dashed, thick] (-1,2,1) -- (0,2,1) -- (0,1,1) -- (-1,1,1) -- cycle;

 \draw[black, dashed, thick] (0,1,-1) -- (0,2,-1) -- (0,2,0) -- (0,1,0) -- cycle;

\draw[black, dashed, thick] (0,2,0) -- (1,2,0) -- (1,1,0) -- (0,1,0) -- cycle;

\draw[black, dashed, thick] (-1,2,0) -- (0,2,0) -- (0,1,0) -- (-1,1,0) -- cycle;
  \node[red] at (-0.2,3.2,-0.3) {$\gamma_1$};
  \node[blue] at (-0.2,3.2,1) {$\gamma_2$};
  \node[black] at (-1,1.2,1.8) {$\nu$};

\end{tikzpicture}
\caption{Second type of cluster $\nu$ s.t. $|\supp~ \nu^+|=4(m-1)-2$}
\label{fig:loops_b}
\end{subfigure}
\caption{Two neighboring Wilson loops $\gamma_1$, $\gamma_2$ in $\Z^3$ with corresponding oriented surfaces $q_1$ and $q_2$. Drawn in black are the two types of vortex cluster $\nu^+$ with a support size of $4(m-1)-2$.}
\label{fig:nuofsecminimsize}
\end{figure}
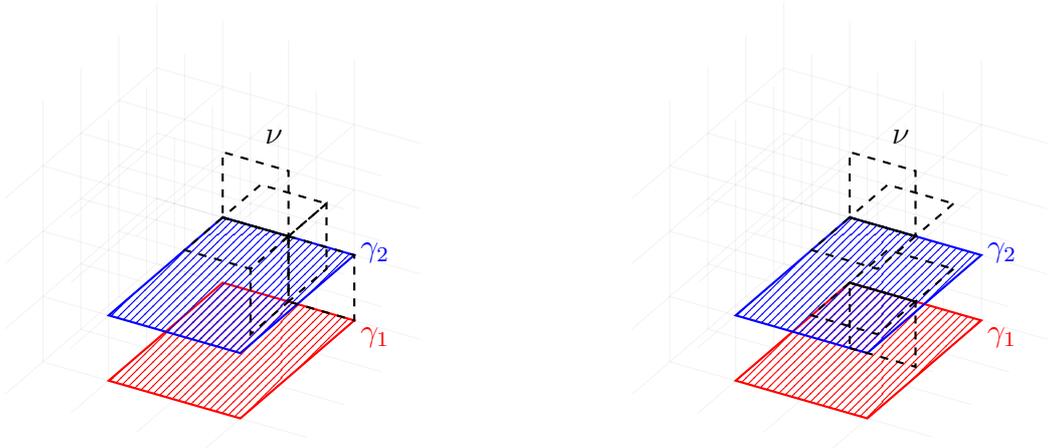
\begin{figure}[ht]
\centering

\begin{tikzpicture}[tdplot_main_coords, scale=0.8]

  \foreach \x in {0,1,2,3}
    \foreach \y in {0,1,2,3}
      \foreach \z in {0,1,2,3}
        \draw[gray, opacity=0.1] (\x,\y,\z) -- (\x+1,\y,\z);

  \foreach \x in {0,1,2,3}
    \foreach \y in {0,1,2,3}
      \foreach \z in {0,1,2,3}
        \draw[gray, opacity=0.1] (\x,\y,\z) -- (\x,\y+1,\z);

  \foreach \x in {0,1,2,3}
    \foreach \y in {0,1,2,3}
      \foreach \z in {0,1,2,3}
        \draw[gray, opacity=0.1] (\x,\y,\z) -- (\x,\y,\z+1);

  \path[pattern=north east lines, pattern color=red]
    (0,1,0) -- (0,3,0) -- (3,3,0) -- (3,1,0) -- cycle; -- cycle;
  
  \draw[red, thick]
    (0,1,0) -- (0,3,0) -- (3,3,0) -- (3,1,0) -- cycle;

  \path[pattern=north east lines, pattern color=blue]
    (0,1,1) -- (0,3,1) -- (3,3,1) -- (3,1,1) -- cycle;

  \draw[blue, thick]
    (0,1,1) -- (0,3,1) -- (3,3,1) -- (3,1,1) -- cycle;

\draw[black, dashed, thick]
  (1,1,1) -- (1,2,1) -- (1,2,2) -- (1,1,2) -- cycle;


\draw[black, dashed, thick] (1,2,1) -- (2,2,1) -- (2,1,1) -- (1,1,1) -- cycle;

\draw[black, dashed, thick] (0,2,1) -- (1,2,1) -- (1,1,1) -- (0,1,1) -- cycle;

 \draw[black, dashed, thick] (1,1,-1) -- (1,2,-1) -- (1,2,0) -- (1,1,0) -- cycle;

\draw[black, dashed, thick] (1,2,0) -- (2,2,0) -- (2,1,0) -- (1,1,0) -- cycle;

\draw[black, dashed, thick] (0,2,0) -- (1,2,0) -- (1,1,0) -- (0,1,0) -- cycle;
  \node[red] at (-0.2,3.2,-0.3) {$\gamma_1$};
  \node[blue] at (-0.2,3.2,1) {$\gamma_2$};
  \node[black] at (0,1.2,1.8) {$\nu$};

\end{tikzpicture}
\caption{Two neighboring Wilson loops $\gamma_1$, $\gamma_2$ in $\Z^3$ with corresponding dashed oriented surfaces $q_1,q_2$. The vortex cluster $\nu^+$ has a support of $4(m-1)-2$, and satisfies $\nu(q_1)=\nu(q_2)=0$.}
\label{fig:thethirdone}
\end{figure}
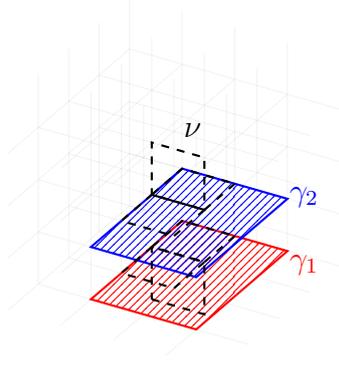
\newline 
The configuration in Figure \ref{fig:loops_a} can't interact with both oriented surfaces for the same reasons as above. For the second configuration, see, e.g., Figure~\ref{fig:loops_b}, one realizes that it can interact with both oriented surfaces $q_i$ and $q_{i+1}$. To see this, let $\nu$ be of this type such that $\supp ~ \nu^+ \cap q_i \neq \emptyset$. Then, by the position of the cluster $\nu$, either $|\nu^+ \cap q_i|=1$ or $|\nu^+ \cap q_i|=2$. Note that it can't be any larger by the definition of the oriented surface $q_i$. If $|\nu^+ \cap q_i|=2$, as in pictured in Figure \ref{fig:thethirdone} (the cluster is not at the boundary of $q_i$), then by definition $\nu(q_i)=0$. Hence one is interested in the case $|\nu^+ \cap q_i|=1$, in which case $\nu(q_i)=1$. In order to obtain $\nu(q_{i+1})$, i.e., $|\nu^+ \cap q_{i+1}|=1$, one sees that there are $|\gamma_i|=|\gamma_{i+1}|$ vortex clusters that satisfy this property : one for every edge $e \in \gamma_i$. This can be observed in Figure~\ref{fig:loops_b}. 

Hence, we obtain 
\begin{equation}
\label{eq:pflemma52firstequation}
\sum_{\substack{\mathcal{V} \in \Xi[\{i,i+1\}]_{4(m-1)-2} \\ n(\mathcal{V})=1}} \Psi_\beta (\mathcal{V}) = |\gamma_i| \myexp(-4(4(m-1)-2)\beta).
\end{equation} 
By Lemma \ref{lem:sizesupportvortex}, the next vortex  $\nu$ of minimal support satisfies \mbox{$|\supp \nu^+|=4(m-1)$}. 
In particular \cite[Lemma 4.7]{firstpaper} proves that there is some constant $C_m>0$, which depends on $m$ but not on the loop size $|\gamma_i|$ such that,
\begin{equation}
\label{eq:pflemma52secondequation}
    \sum_{\substack{\mathcal{V} \in \Xi[\{i,i+1\}]_{4(m-1)^+} \\ n(\mathcal{V})=1}} |\Psi_\beta (\mathcal{V})| \leq |\gamma_1| C_m \myexp(-4(4(m-1)\beta)).
\end{equation}
The claim follows by combining \eqref{eq:pflemma52firstequation} and \eqref{eq:pflemma52secondequation}.    
\end{proof}
\begin{proof}[Proof of Lemma \ref{lem:firstordersecondsummand}]
This proof will follow a similar path as the proof of ~Lemma~\ref{lem:firstordertwosummands}. \newline 
Let $\mathcal{V} \in \Xi$ such that $n(\mathcal{V}) = 2$. In this case, note that $\mathcal{U}(\mathcal{V})=-1$. Such a cluster can have two forms, either $\mathcal{V}=\{\nu,\nu\}$ or $\mathcal{V}=\{\nu_1,\nu_2\}, ~ \nu_1 \neq \nu_2$, where $\nu,\nu_1,\nu_2 \in \Lambda(B_N)$. \newline 
If $\mathcal{V}=\{\nu,\nu\}$, one deduces that for any 2-form $q$,
\begin{align*}
&\nu(q) = 0 ~~ \text{implies} ~~ \mathcal{V}(q)=0, \\
~~ \text{while} ~ &\nu(q)=1 ~~ \text{implies} ~~\mathcal{V}(q)= \nu(q)+\nu(q) =0,
\end{align*}
which implies that $\mathcal{V} \notin \Xi[\{i,i+1\}]$.
Hence, if $\mathcal{V} \in \Xi[\{i,i+1\}]$, we must have $\mathcal{V}=\{\nu_1,\nu_2\}$. 
Here by definition, as $n(\mathcal{V})=2$ implies $ \mathcal{U}(\mathcal{V})=-1$, equation \eqref{eq:Decompositioncorracti} gives that
\[
\Psi_\beta(\mathcal{V})= - \myexp\big(-4\beta(|\supp ~ \nu_1^+|+|\supp ~ \nu_2^+|)\big).
\]
In the case in which both vortex clusters have minimal support, which mean, $|\supp\nu_1^+| = |\supp\nu_2^+| = 2(m-1)$, there are $|\gamma_i|$ translations of $\mathcal{V} \in \Xi[\{i,i+1\}]$ for which $\mathcal{V}(q_i)=\mathcal{V}(q_{i+1})=1$. Namely one $\mathcal{V}=\{\nu_1,\nu_2\}$ for each $e \in \gamma_i$ as pictured in ~Figure~\ref{fig:anewfigurewhosenameidk} with the associated green edge $e \in \gamma_1$. Observe that if $\nu'_1$ and $\nu'_2$ are not on top of each other but still interacting with $q_i$ and $q_{i+1}$, then $\mathcal{V}=\{\nu'_1,\nu'_2\} \notin \Xi$, as $\nu'_1 \nsim \nu'_2$ in the graph $\mathcal{G}_2$.
By the same argument as in the proof of Lemma~\ref{lem:firstordertwosummands}, there is no other possibility in which both $\nu_1$ and $\nu_2$ have minimal support and satisfy $\mathcal{V}(q_i)=\mathcal{V}(q_{i+1})=1$. Hence, one obtains that
\[
\sum_{\substack{\mathcal{V} \in \Xi[\{i,i+1\}]_{4(m-1)} \\ n(\mathcal{V})=2}} \Psi_\beta (\mathcal{V}) = - |\gamma_i| \myexp(-4(4(m-1))\beta).
\]
\begin{figure}[ht]
\centering

\begin{tikzpicture}[tdplot_main_coords, scale=1]

  \foreach \x in {0,1,2,3}
    \foreach \y in {0,1,2,3}
      \foreach \z in {0,1,2,3}
        \draw[gray, opacity=0.1] (\x,\y,\z) -- (\x+1,\y,\z);

  \foreach \x in {0,1,2,3}
    \foreach \y in {0,1,2,3}
      \foreach \z in {0,1,2,3}
        \draw[gray, opacity=0.1] (\x,\y,\z) -- (\x,\y+1,\z);

  \foreach \x in {0,1,2,3}
    \foreach \y in {0,1,2,3}
      \foreach \z in {0,1,2,3}
        \draw[gray, opacity=0.1] (\x,\y,\z) -- (\x,\y,\z+1);

  \path[pattern=north east lines, pattern color=red]
    (0,1,0) -- (0,3,0) -- (3,3,0) -- (3,1,0) -- cycle; -- cycle;
  
  \draw[red, thick]
    (0,1,0) -- (0,3,0) -- (3,3,0) -- (3,1,0) -- cycle;

  \path[pattern=north east lines, pattern color=blue]
    (0,1,1) -- (0,3,1) -- (3,3,1) -- (3,1,1) -- cycle;

  \draw[blue, thick]
    (0,1,1) -- (0,3,1) -- (3,3,1) -- (3,1,1) -- cycle;

\draw[gray, dashed, thick]
  (0,1,1) -- (0,2,1) -- (0,2,2) -- (0,1,2) -- cycle;

 \draw[gray, dashed, thick] (0,1,0) -- (0,2,0) -- (0,2,1) -- (0,1,1) -- cycle;

\draw[gray, dashed, thick] (0,2,1) -- (1,2,1) -- (1,1,1) -- (0,1,1) -- cycle;

\draw[gray, dashed, thick] (-1,2,1) -- (0,2,1) -- (0,1,1) -- (-1,1,1) -- cycle;
\draw[black, dashed, thick]
  (0,1,0) -- (0,2,0) -- (0,2,1) -- (0,1,1) -- cycle;

 \draw[black, dashed, thick] (0,1,-1) -- (0,2,-1) -- (0,2,0) -- (0,1,0) -- cycle;

\draw[black, dashed, thick] (0,2,0) -- (1,2,0) -- (1,1,0) -- (0,1,0) -- cycle;

\draw[black, dashed, thick] (-1,2,0) -- (0,2,0) -- (0,1,0) -- (-1,1,0) -- cycle;
  \node[red] at (-0.2,3.2,-0.3) {$\gamma_1$};
  \node[blue] at (-0.2,3.2,1) {$\gamma_2$};
  \node[gray] at (-1,1.2,1.8) {$\nu_1$};
  \node[black] at (-0.2,3,0.4) {$\nu_2$};
   \node[ForestGreen] at (-0.2,3.4,0.13) {$e$};
  \draw[green!70!black, ultra thick]
    (0,2,0) -- (0,1,0);
\end{tikzpicture}
\caption{Two neighboring Wilson loops $\gamma_1$, $\gamma_2$ in $3$ with corresponding dashed oriented surfaces $q_1,q_2$. Each vortex cluster $\nu_1^+,\nu_2^+$ has a support size of $2(m-1)$.}
\label{fig:anewfigurewhosenameidk}
\end{figure}
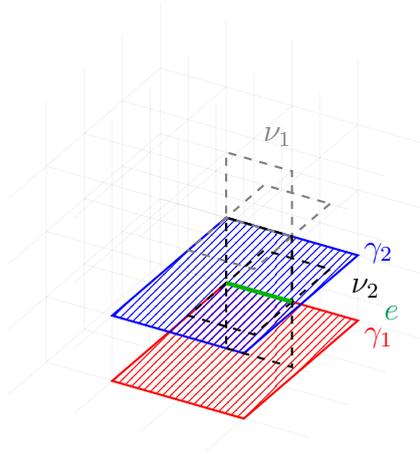 
\newline
Now consider any other $\mathcal{V}=\{\nu_1,\nu_2\}$ such that $\mathcal{V}(\gamma_i)= \mathcal{V}(\gamma_{i+1}) = 1$. Then by~Lemma~\ref{lem:sizesupportvortex} it holds that
\[
|\mathcal{V}| = |\supp ~ \nu_1^+|+|\supp ~ \nu_2^+| \geq 2(m-1)+4(m-1)-2 = 6(m-1)-2.
\]
Similarly if $n(\mathcal{V}) \geq 3$, one obtains
\[
|\mathcal{V}| = \sum_{\nu \in \mathcal{V}} n_\mathcal{V}(\nu) |\supp ~ \nu| \geq 3 \cdot 2(m-1) = 6(m-1)  \geq 6(m-1)-2.
\] 
\newline 
For $\mathcal{V} \in \Xi$ such that $|\mathcal{V}| \geq 6(m-1)-2$, we apply as previously \cite[Lemma 4.7]{firstpaper}, which gives a constant $\tilde{C}_m$ which only depends on $m$ such that
\[
\sum_{\substack{\mathcal{V} \in \Xi[\{i,i+1\}]_{6(m-1)-2^+} \\ n(\mathcal{V}) \geq 2}} \Psi_\beta (\mathcal{V}) \leq |\gamma_1| \tilde{C}_m \myexp\big(-4(6(m-1)-2)\beta\big).
\]
This concludes the proof.
\end{proof}
\begin{Lemma}
    \label{lem:restorders}
    Consider Ising lattice gauge theory with parameter $\beta \geq \beta_0$ and the set of loops $(\gamma_i)_{i \in [n]}$ from Definition \ref{Def:Wilsonloops}. Now, consider any set $I \subset [n]$ such that
\[
I \in \mathcal{I}_n:= \mathcal{P}([n]) \setminus \{\{1\},...,\{n\},\{1,2\}, \{2,3\}, ...,\{n-1,n\}\},
\]
where $\mathcal{P}$ denotes the power set.
Then, for the loops $\gamma_1,...,\gamma_n$ from Definition \ref{Def:Wilsonloops}, we have
    \begin{equation*}
        \Psi_\beta[I] = |\gamma_1|   \mathcal{O}\big(\myexp{(-4(4(m-1))\beta)}\big),
    \end{equation*}
    where the term $\mathcal{O}((\myexp{(-4(4(m-1))\beta)})$ is independent of the loop's size $|\gamma_1|$, and the asymptotics is in $\beta$. 
\end{Lemma}
\begin{proof}
    Let $I \in \mathcal{I}_n$ and $\mathcal{V} \in \Xi[I]$.
    Then $\max I - \min I \geq 2$, hence, $\mathcal{V}$ is a vortex cluster which connects oriented surfaces that are separated by a distance of at least two (in the direction $\tilde{e}_3$) plaquettes from each other. Hence, by a similar argumentation as in the proof of Lemma \ref{lem:firstordertwosummands}, we deduce the following.  If $n(\mathcal{V})=1$, it must hold that  $|\supp ~ \nu^+| \geq 4(m-1)$ by Lemma \ref{lem:sizesupportvortex}, as it is not possible for the vortex clusters satisfying $|\supp ~ \nu^+| \in \{2(m-1), 4(m-1)-2\}$ to connect oriented surfaces with a distance greater than two from each other. One can observe this in Figure \ref{fig:threeloops1} for the case $|\supp ~ \nu^+| = 4(m-1)-2$, in which case one is just able to combine the two oriented surfaces $q_2$ and $q_3$. 
 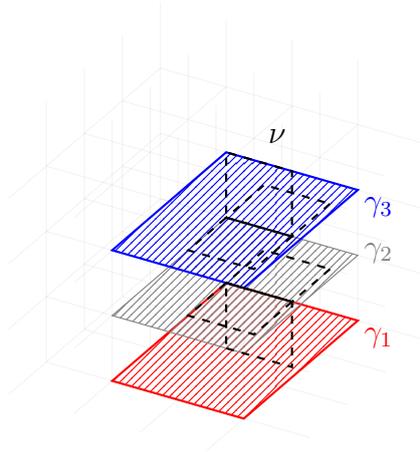
\begin{figure}[ht]
\centering
\begin{tikzpicture}[tdplot_main_coords, scale=1]

  \foreach \x in {0,1,2,3}
    \foreach \y in {0,1,2,3}
      \foreach \z in {0,1,2,3}
        \draw[gray, opacity=0.1] (\x,\y,\z) -- (\x+1,\y,\z);

  \foreach \x in {0,1,2,3}
    \foreach \y in {0,1,2,3}
      \foreach \z in {0,1,2,3}
        \draw[gray, opacity=0.1] (\x,\y,\z) -- (\x,\y+1,\z);

  \foreach \x in {0,1,2,3}
    \foreach \y in {0,1,2,3}
      \foreach \z in {0,1,2,3}
        \draw[gray, opacity=0.1] (\x,\y,\z) -- (\x,\y,\z+1);

  \path[pattern=north east lines, pattern color=red]
    (0,1,0) -- (0,3,0) -- (3,3,0) -- (3,1,0) -- cycle; -- cycle;
  
  \draw[red, thick]
    (0,1,0) -- (0,3,0) -- (3,3,0) -- (3,1,0) -- cycle;

  \path[pattern=north east lines, pattern color=gray]
    (0,1,1) -- (0,3,1) -- (3,3,1) -- (3,1,1) -- cycle;

  \draw[gray]
    (0,1,1) -- (0,3,1) -- (3,3,1) -- (3,1,1) -- cycle;

  \path[pattern=north east lines, pattern color=blue]
    (0,1,2) -- (0,3,2) -- (3,3,2) -- (3,1,2) -- cycle;

  \draw[blue, thick]
    (0,1,2) -- (0,3,2) -- (3,3,2) -- (3,1,2) -- cycle;

\draw[black, dashed, thick]
  (0,1,1) -- (0,2,1) -- (0,2,2) -- (0,1,2) -- cycle;


\draw[black, dashed, thick] (0,2,1) -- (1,2,1) -- (1,1,1) -- (0,1,1) -- cycle;

\draw[black, dashed, thick] (-1,2,1) -- (0,2,1) -- (0,1,1) -- (-1,1,1) -- cycle;

 \draw[black, dashed, thick] (0,1,-1) -- (0,2,-1) -- (0,2,0) -- (0,1,0) -- cycle;

\draw[black, dashed, thick] (0,2,0) -- (1,2,0) -- (1,1,0) -- (0,1,0) -- cycle;

\draw[black, dashed, thick] (-1,2,0) -- (0,2,0) -- (0,1,0) -- (-1,1,0) -- cycle;
  \node[red] at (-0.2,3.2,-0.3) {$\gamma_1$};
  \node[blue] at (-0.2,3.2,1.7) {$\gamma_3$};
  \node[black] at (-1,1.2,1.8) {$\nu$};
\node[gray] at (-0.2,3.2,1) {$\gamma_2$};
\end{tikzpicture}
\caption{Three neighboring Wilson loops $\gamma_1, \gamma_2,\gamma_3$ in $\Z^3$. The vortex cluster $\nu^+$ has a support of $4(m-1)-2$ and can only interact with $q_2$ and $q_3$.}
\label{fig:threeloops1}
\end{figure}
\newline 
 Now, in the case $n(\mathcal{V}) \ge 2$. One has that
    \[
    |\mathcal{V}| \geq 2(m-1)+2(m-1) = 4(m-1).
    \] 
    Hence \cite[Lemma 4.7]{firstpaper} gives the following inequality,
    \[
    \sum_{\mathcal{V} \in \Xi[I]} \Psi_\beta (\mathcal{V}) \leq |\gamma_1|\hat{C}_m \myexp(-4(4(m-1))\beta),
    \]
    for some constant $\hat{C}_m$ only depending on the dimension $m$.
\end{proof}
\subsection{Proof of Theorem \ref{thm:1}}
Recall the construction of the Wilson loops $(\gamma_i)_{i \in [n]}$ from Definition \ref{Def:Wilsonloops}. 
In order to prove Theorem \ref{thm:1}, we will prove that
\begin{equation*}
\label{eq:Unmaintoprove}
   |V_n^-(W_\gamma)| < 1.
\end{equation*}
Using \eqref{eq:decompv} the claim that $U_n^W(W_\gamma) >0$ will directly follow. To this end, we define the following quantity, which will be used in the proof.
\begin{equation}
\label{eq:DefSn}
\mathscr{S}(n) := \sum_{\substack{\mathfrak{P} \in \mathfrak{F}_n \setminus\{[n]\}\\ |\mathfrak{P}| ~~ \text{even}}} (|\mathfrak{P}|-1)!
\end{equation}
We can compute $\mathscr{S}(n)$ recursively, see, e.g., \cite[Chapter 6]{bookstirling}.  
However, for the following proof, we simply need to know that $\mathscr{S}(n)$ is finite for all $n$, which is clear by the definition. 
\begin{proof}[Proof of Theorem \ref{thm:1}]
    By Lemma \ref{lem:firstorder}, 
    there exist functions, $(f_i)_{i \in [n-1]}$, independent of the loop's size $|\gamma_1|$, such that for any $i \in [n-1]$,
\begin{equation}
\label{eq:pfeq1}
		\frac{\Psi_\beta[\{i,i+1\}]}{\myexp{(-4(4(m-1)-2)\beta)}} = |\gamma_1| (1+f_i(\beta)), ~~ \text{where} ~~ f_i(\beta) = \mathcal{O}(\myexp(-8\beta)). 
\end{equation}
At the same time, by Lemma \ref{lem:restorders}, there exist functions, $(f_I)_{I \in \mathcal{I}_n}$, independent of the loop's size, such that for any $I \in \mathcal{I}_n$,
\begin{equation}
\label{eq:pfeq2}
\frac{\Psi_\beta[I] }{\myexp{(-4(4(m-1)-2)\beta)}} = |\gamma_1|  f_I(\beta), ~~~ f_I(\beta) = \mathcal{O}(\myexp(-8\beta)), 
\end{equation}
where the asymptotics are for large values of $\beta$. \newline
Now, as $(f_i)_{i \in [n-1]}$ and $(f_I)_{I \in \mathcal{I}_n}$ do not depend on the loop's size $|\gamma_1|$, we can define $\beta_{n,m}^*$ as the smallest value such that for all $\beta \geq \beta_{n,m}^*$  and for all $i \in [n-1]$ the following is satisfied:  
\begin{equation}
\tag{A}
\label{eq:endA}
2\Psi_\beta[\{i,i+1\}] - \sum_{I \in \mathcal{I}_n} n|\Psi_\beta[I] | \geq \frac{|\gamma_1|}{2} \myexp(-4(4(m-1)-2)\beta).
\end{equation} 
To such $\beta$, we make the loops $(\gamma_i)_{i \in [n]}$ large enough so that
\begin{equation}
\tag{B}
\label{eq:endB}
|\gamma_1| \myexp(4(4(m-1)-2)\beta) > \mylog(\mathscr{S}(n)).
\end{equation}
Consider now any $\mathfrak{P} \in \mathfrak{F}_n$ such that $|\mathfrak{P}|$ is even. Then, by Corollary \ref{cor:firstcorollary}, we have that
\[
 b_{\mathfrak{P}} = - (\mathfrak{P}-1)! \myexp\Big(-2 \sum_{I \subset [n]} (c_{I,\mathfrak{P}} - \mathbbm{1}\{|I| ~ \text{odd}\}) \Psi[I] \Big).
\]
From Lemma \ref{lem:comesatend}, we know that
\begin{equation}
\tag{C}
\label{eq:endC}
c_{I,\mathfrak{P}} \leq n ~ \text{for any} ~ I \subseteq [n], ~ \text{and} ~ c_{\{i,i+1\},\mathfrak{P}} \neq 0 \iff c_{\{i,i+1\},\mathfrak{P}} = 2. 
\end{equation}
Now, consider any element of the partition $\mathfrak{P} \in \mathfrak{F}_n \setminus\{[n]\}$. By Lemma \ref{lem:comesatend}(f), we know that there exists $i \in [n-1]$ such that $c_{\{i,i+1\},\mathfrak{P}}>0$.  
As $\Psi_{\beta}[\{i,i+1\}] > 0$ for any $i \in [n-1]$ by \eqref{eq:endA}, we have that,
\begin{align*}
&~\sum_{I \subset [n]} (c_{I,\mathfrak{P}}-\mathbbm{1}\{|I| ~ \text{odd}\}) \Psi_\beta[I] \\
=&\sum_{i \in [n-1]} c_{\{i,i+1\},\mathfrak{P}} \Psi_\beta[\{i,i+1\}]  + \sum_{I \in \mathcal{I}_n} (c_{I,\mathfrak{P}}-\mathbbm{1}\{|I| ~ \text{odd}\}) \Psi_\beta[I] \\ \geq & ~~2\Psi_\beta[\{i,i+1\}] - \sum_{I \in \mathcal{I}_n} n|\Psi_\beta[I]| > \frac{1}{2}\mylog(\mathscr{S}(n)).
\end{align*}
Where the first equality comes from Lemma \ref{lem:comesatend}(c) and the first inequality from \eqref{eq:endC}. Finally, the last inequality comes from the combination of \eqref{eq:endA} and \eqref{eq:endB}.
It implies that,
\begin{equation}
\tag{D}
\label{eq:endD}
\myexp \Big(- 2 \sum_{I \subset [n]} (c_{I,\mathfrak{P}}-\mathbbm{1}\{|I| ~ \text{odd}\}) \Psi_\beta[I] \Big) < \frac{1}{\mathscr{S}(n)}.
\end{equation}
As $\mathfrak{P} \in \mathfrak{F}_n, ~ |\mathfrak{P}| ~ \text{even}$, was arbitrary, we apply inequality \eqref{eq:endD} to all such $\mathfrak{P}$ in the expression of $V_n^-(W_\gamma)$ from Corollary \ref{cor:firstcorollary} to obtain
\[
|V_n^-(W_\gamma)| = \sum_{\substack{\mathfrak{P} \in \mathfrak{F}_n \\ |\mathfrak{P}| ~ \text{even}}} |b_{\mathfrak{P}}| < \sum_{\substack{\mathfrak{P} \in \mathfrak{F}_n\setminus\{[n]\}  \\ |\mathfrak{P}| ~ \text{even}}} (|\mathfrak{P}|-1)! \frac{1}{\mathscr{S}(n)} = \frac{1}{\mathscr{S}(n)} \mathscr{S}(n) = 1.
\]
Here, the inequality comes from \eqref{eq:endD} and the second equality from the definition of $\mathscr{S}(n)$. This concludes the proof by \eqref{eq:decompv}.
\end{proof}
\newpage 

\appendix
\section{Proof of Lemma \ref{lem:abouttheloopdec}}
\label{App:A}
 We seek the smallest number of edges $n$ such that $e_1,...,e_n$ form a loop $\gamma$ which can be decomposed in at least two distinct ways: $\gamma=\gamma_1+\gamma_2=\gamma_1'+\gamma_2'$, where $\gamma_1,\gamma_2 ~\text{and} ~\gamma_1',\gamma_2'$ denote different closed loops. In the case $m \geq 3$, one needs $n \geq 10$ to achieve this, notably by taking advantage of the third dimension, as one can observe in the proof of Theorem \ref{thm:2} (d). We will prove that in the two-dimensional case, one needs $n \geq 16$ edges. 
  Note, that for two distinct loops $\gamma_1,\gamma_2$, we say that $\gamma_2$ \textit{touches} an edge $e \in \gamma_1$, if there exists $\tilde{e} \in \gamma_2$ such that $\partial e \cap \partial \tilde{e} \neq \emptyset$. If $\gamma_2$ touches an edge $e \in \gamma_1$, then we say that $\gamma_2$ \textit{crosses} $\gamma_1$ if there is no corner of $\gamma_1$ at the intersection point. 
  With these definitions established, we can prove Lemma \ref{lem:abouttheloopdec}. 
\begin{proof}[Proof of Lemma \ref{lem:abouttheloopdec}]
    Denote the first decomposition of $\gamma$ by $\gamma_1+\gamma_2$, where $\gamma_1$ and $\gamma_2$ are two closed loops. Then, the loop $\gamma_2$ must be situated around $\gamma_1$ in the following sense: there exists $e \in \gamma_1$ such that either one or both boundary points of $e$ are touched by  $\gamma_2$. 
    Such an edge $e$ must exist to be able to "leave" $\gamma_1$ and be in another decomposition of $\gamma$, i.e., $e\in \gamma_1'~ \text{or}~\gamma_2'$.
    \newline 
    \textit{Case 1.} Assume there is $e \in \gamma_1$ such that both of its boundaries are traversed by $\gamma_2$, as depicted in Figure \ref{fig:PfFig0}. 
    \begin{figure}[H]
    \centering
\tdplotsetmaincoords{60}{120}

\begin{tikzpicture}[scale=0.55, every node/.style={font=\small}]
  \draw[step=1cm, very thin, gray!40] (0,0) grid (5,5);

\draw[thick, red] (2,2)-- (3,2);
  \node[red] at (2.5,1.6) {$e$};
\draw[dashed, red] (2,2) -- (2,0);
\draw[dashed, red] (3,2) -- (3,0);
\node[red] at (3.3,0.9) {$\gamma_1$};
\draw[thick,blue] (4,2) -- (3,2);
\draw[thick,blue] (3,2) -- (3,3);
\draw[dashed,blue] (2,3) -- (3,3);
\draw[thick,blue] (2,3) -- (2,2);
\draw[thick,blue] (2,2) -- (1,2);
\draw[dashed,blue] (1,3) -- (1,2);
\draw[dashed,blue] (4,2) -- (4,3);
\node[blue] at (3.5,2.5) {$\gamma_2$};
\draw (2,2) node[fill=black!60,circle,inner sep=1.2pt] {};
\draw (3,2) node[fill=black!60,circle,inner sep=1.2pt] {};

  \foreach \x in {0,...,5}
    \foreach \y in {0,...,5}
      \fill[gray!40] (\x,\y) circle (0.03);

\end{tikzpicture}

    \caption{Configuration type in case 1, where both boundaries of $e \in \gamma_1$ get touched by $\gamma_2$.}
    \label{fig:PfFig0}
\end{figure}
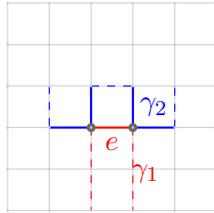
    Then one searches for the smallest loop $\gamma=\gamma_1+\gamma_2$ that satisfies the following property. There is an edge $e$ in the loop $\gamma_1$ such that a second loop $\gamma_2$ goes through both boundaries of $e$ without containing $e$. Clearly $\gamma_1$ contains at least 4 edges, which is the smallest loop, and $\gamma_2$ contains at least 6 edges. In this case we need to check that for all $|\gamma_1| \in \{4,6,8\}$ it holds that $|\gamma_1|+|\gamma_2|\geq 16$. 
    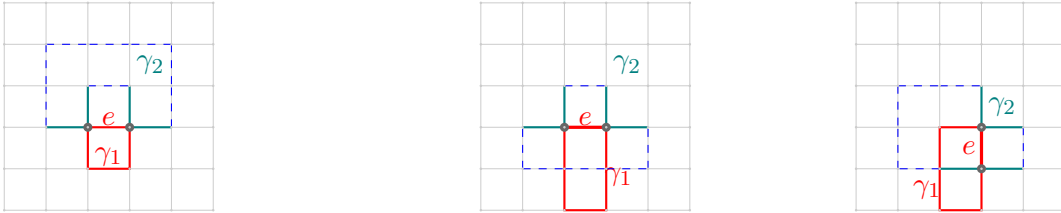
\begin{figure}[H]
    \centering
    \begin{subfigure}[b]{0.32\textwidth}
        \begin{tikzpicture}[scale=0.55, every node/.style={font=\small}]
  \draw[step=1cm, very thin, gray!40] (0,0) grid (5,5);

\draw[thick, red] (2,2) rectangle (3,1);
\node[red] at (2.5,2.2) {$e$};
\node[red] at (2.5,1.3) {$\gamma_1$};
\draw[thick,teal] (4,2) -- (3,2);
\draw[thick,teal] (3,2) -- (3,3);
\draw[dashed,blue] (2,3) -- (3,3);
\draw[thick,teal] (2,3) -- (2,2);
\draw[thick,teal] (2,2) -- (1,2);
\draw[dashed,blue] (1,3) -- (1,2);
\draw[dashed,blue] (4,2) -- (4,3);
\draw[dashed,blue] (4,4) -- (4,3);
\draw[dashed,blue] (4,4) -- (3,4);
\draw[dashed,blue] (3,4) -- (2,4);
\draw[dashed,blue] (1,4) -- (2,4);
\draw[dashed,blue] (1,4) -- (1,3);
\node[teal] at (3.5,3.5) {$\gamma_2$};
\draw (2,2) node[fill=black!60,circle,inner sep=1.2pt] {};
\draw (3,2) node[fill=black!60,circle,inner sep=1.2pt] {};

  \foreach \x in {0,...,5}
    \foreach \y in {0,...,5}
      \fill[gray!40] (\x,\y) circle (0.03);

\end{tikzpicture}

    \caption{Configuration such that $|\gamma_1|=4$.}
    \label{fig:PfFig1a}
    \end{subfigure}
\hfill
\begin{subfigure}[b]{0.32\textwidth}
    \centering

\begin{tikzpicture}[scale=0.55, every node/.style={font=\small}]
  \draw[step=1cm, very thin, gray!40] (0,0) grid (5,5);

\draw[thick, red] (2,2) rectangle (3,0);
\draw[very thick, red] (2,2) -- (3,2);
\node[red] at (2.5,2.2) {$e$};
\node[red] at (3.3,0.8) {$\gamma_1$};
\draw[thick,teal] (4,2) -- (3,2);
\draw[thick,teal] (3,2) -- (3,3);
\draw[dashed,blue] (2,3) -- (3,3);
\draw[thick,teal] (2,3) -- (2,2);
\draw[thick,teal] (2,2) -- (1,2);
\draw[dashed,blue] (1,1) -- (1,2);
\draw[dashed,blue] (1,1) -- (2,1);
\draw[dashed,blue] (2,1) -- (3,1);
\draw[dashed,blue] (3,1) -- (4,1);
\draw[dashed,blue] (4,1) -- (4,2);
\node[teal] at (3.5,3.5) {$\gamma_2$};
\draw (2,2) node[fill=black!60,circle,inner sep=1.2pt] {};
\draw (3,2) node[fill=black!60,circle,inner sep=1.2pt] {};

  \foreach \x in {0,...,5}
    \foreach \y in {0,...,5}
      \fill[gray!40] (\x,\y) circle (0.03);

\end{tikzpicture}

    \caption{First type of configuration such that $|\gamma_1|=6$.}
    \label{fig:PfFig1b}
\end{subfigure}
\begin{subfigure}[b]{0.32\textwidth}
    \centering
\begin{tikzpicture}[scale=0.55, every node/.style={font=\small}]
  \draw[step=1cm, very thin, gray!40] (0,0) grid (5,5);

\draw[thick, red] (2,2) rectangle (3,0);
\draw[very thick, red] (3,1) -- (3,2);
\node[red] at (2.7,1.5) {$e$};
\node[red] at (1.7,0.5) {$\gamma_1$};
\draw[thick,teal] (4,2) -- (3,2);
\draw[thick,teal] (3,2) -- (3,3);
\draw[thick,teal] (3,1) -- (2,1);
\draw[thick,teal] (3,1) -- (4,1);
\draw[dashed,blue] (1,3) -- (1,2);
\draw[dashed,blue] (4,2) -- (4,1);
\draw[dashed,blue] (2,1) -- (1,1);
\draw[dashed,blue] (1,1) -- (1,2);
\draw[dashed,blue] (1,3) -- (2,3);
\draw[dashed,blue] (2,3) -- (3,3);
\node[teal] at (3.5,2.5) {$\gamma_2$};
\draw (3,1) node[fill=black!60,circle,inner sep=1.2pt] {};
\draw (3,2) node[fill=black!60,circle,inner sep=1.2pt] {};

  \foreach \x in {0,...,5}
    \foreach \y in {0,...,5}
       \fill[gray!40] (\x,\y) circle (0.03);

\end{tikzpicture}

    \caption{Second type of configuration such that $|\gamma_1|=6$}
    \label{fig:PfFig1c}
\end{subfigure}
\caption{The subfigures illustrate the types of configurations in Case 1, when $|\gamma_1| \in \{4,6\}$. Each highlight the edge $e \in \gamma_1$  touched twice by $\gamma_2$. The teal edges are those that $\gamma_2$ needs, to cross the boundaries of $e$. The dashed blue edges show a shortest path to close $\gamma_2$. In all three cases, $|\gamma_1|+|\gamma_2| \geq 16$.}
\end{figure}
   If $|\gamma_1|=4$, then one has no other choice than having the four teal edges in Figure~\ref{fig:PfFig1a}, which will be edges in $\gamma_2$. In this case, one needs at least $8$ additional edges (blue dashed) to connect the teal edges into a loop, which gives $|\gamma_2| \geq 12$. Hence, $|\gamma_1|+|\gamma_2| \geq 16$. \newline 
    If $|\gamma_1|=6$, and $e \in \gamma_1$ is located as in Figure \ref{fig:PfFig1b}, then $|\gamma_2| \geq 10$ with a similar argumentation as in Figure \ref{fig:PfFig1a}. If $e \in \gamma_1$ appears as in Figure \ref{fig:PfFig1c}, then $\gamma_2$ must contain the four teal edges. The smallest configuration that closes the loop yields $|\gamma_2| = 10$. In both cases, $|\gamma_1|+|\gamma_2| \geq 16$. 
    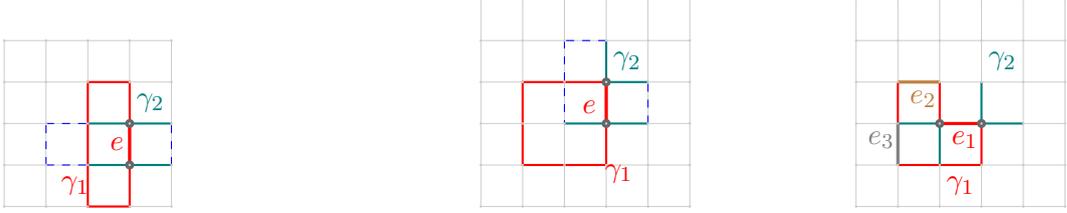
\begin{figure}[H]
    \centering
    \begin{subfigure}[b]{0.32\textwidth}
\begin{tikzpicture}[scale=0.55, every node/.style={font=\small}]
  \draw[step=1cm, very thin, gray!40] (0,0) grid (4,4);

\draw[thick, red] (2,3) rectangle (3,0);
\draw[very thick, red] (3,1) -- (3,2);
\node[red] at (2.7,1.5) {$e$};
\node[red] at (1.7,0.5) {$\gamma_1$};
\draw[thick,teal] (4,1) -- (3,1);
\draw[thick,teal] (3,1) -- (2,1);
\draw[thick,teal] (3,2) -- (2,2);
\draw[thick,teal] (3,2) -- (4,2);
\draw[dashed,blue] (4,1) -- (4,2);
\draw[dashed,blue] (1,2) -- (2,2);
\draw[dashed,blue] (1,1) -- (2,1);
\draw[dashed,blue] (1,1) -- (1,2);
\node[teal] at (3.5,2.5) {$\gamma_2$};
\draw (3,2) node[fill=black!60,circle,inner sep=1.2pt] {};
\draw (3,1) node[fill=black!60,circle,inner sep=1.2pt] {};

  \foreach \x in {0,...,4}
    \foreach \y in {0,...,4}
       \fill[gray!40] (\x,\y) circle (0.03);

\end{tikzpicture}

    \caption{First configuration type such that $|\gamma_1|=8$.}
    \label{fig:PfFig2a}
    \end{subfigure}
\hfill
\begin{subfigure}[b]{0.32\textwidth}
    \centering

\begin{tikzpicture}[scale=0.55, every node/.style={font=\small}]
  \draw[step=1cm, very thin, gray!40] (0,0) grid (5,5);

\draw[thick, red] (1,1) rectangle (3,3);
\draw[very thick, red] (3,3) -- (3,2);
\node[red] at (2.6,2.4) {$e$};
\node[red] at (3.3,0.8) {$\gamma_1$};
\draw[thick,teal] (4,2) -- (3,2);
\draw[thick,teal] (3,2) -- (2,2);
\draw[thick,teal] (4,3) -- (3,3);
\draw[thick,teal] (3,3) -- (3,4);
\draw[dashed,blue] (4,2) -- (4,3);
\draw[dashed,blue] (3,4) -- (2,4);
\draw[dashed,blue] (2,4) -- (2,2);
\node[teal] at (3.5,3.5) {$\gamma_2$};
\draw (3,3) node[fill=black!60,circle,inner sep=1.2pt] {};
\draw (3,2) node[fill=black!60,circle,inner sep=1.2pt] {};

  \foreach \x in {0,...,5}
    \foreach \y in {0,...,5}
      \fill[gray!40] (\x,\y) circle (0.03);

\end{tikzpicture}

    \caption{Second type of configuration such that $|\gamma_1|=8$.}
    \label{fig:PfFig2b}
\end{subfigure}
\begin{subfigure}[b]{0.32\textwidth}
    \centering
\begin{tikzpicture}[scale=0.55, every node/.style={font=\small}]
  \draw[step=1cm, very thin, gray!40] (0,0) grid (5,5);

\node[red] at (2.6,1.6) {$e_1$};
\node[red] at (2.5,0.5) {$\gamma_1$};
\node[brown] at (1.6,2.6) {$e_2$};
\node[gray] at (0.6,1.6) {$e_3$};
\draw[thick,red] (3,1) -- (2,1);
\draw[thick,red] (2,1) -- (1,1);
\draw[very thick,red] (3,2) -- (2,2);
\draw[thick,red] (2,2) -- (2,3);
\draw[thick,red] (3,1) -- (3,2);
\draw[very thick,brown] (2,3) -- (1,3);
\draw[very thick,gray] (1,2) -- (1,1);
\draw[thick,red] (1,3) -- (1,2);
\draw[thick, teal] (3,2) -- (3,3);
\draw[thick, teal] (3,2) -- (4,2) ;
\draw[thick, teal] (2,2) -- (2,1) ;
\draw[thick, teal] (2,2) -- (1,2);
\node[teal] at (3.5,3.5) {$\gamma_2$};
\draw (2,2) node[fill=black!60,circle,inner sep=1.2pt] {};
\draw (3,2) node[fill=black!60,circle,inner sep=1.2pt] {};

  \foreach \x in {0,...,5}
    \foreach \y in {0,...,5}
       \fill[gray!40] (\x,\y) circle (0.03);

\end{tikzpicture}

    \caption{Third type of configuration such that $|\gamma_1|=8$}
    \label{fig:PfFig2c}
\end{subfigure}
\caption{Types of configurations in Case 1 when $|\gamma_1| =8$. Here one sees the edge $e \in \gamma_1$ that is traversed twice by $\gamma_2$. The teal edges of $\gamma_2$ are the necessary edges in $\gamma_2$ to cross the boundaries of $e$, while the dashed blue edges correspond to a shortest path to close $\gamma_2$. In particular one has $|\gamma_1|+|\gamma_2| \geq 16$ in the three cases.}
\end{figure}

    Now, in the case $|\gamma_1| = 8$, the loop $\gamma_1$ can have three shapes. In the first case, pictured in Figure \ref{fig:PfFig2a}, one must have $|\gamma_2| \geq 8$ with a similar argumentation as in Figure \ref{fig:PfFig1c}. In the case where $\gamma_1$ is a square of 8 edges, as depicted in Figure \ref{fig:PfFig2b}, then $\gamma_2$ needs to contain the four teal edges marked in Figure \ref{fig:PfFig2b}, and one deduces that $|\gamma_2| \geq 8$ in order to connect the teal edges into a loop. The last possible shape of $\gamma_1$ is pictured in Figure \ref{fig:PfFig2c}. Here, if the edge of $\gamma_1$ that is touched twice is $e_2$ (marked in brown), then one has the same case as in Figure \ref{fig:PfFig1a}, where $|\gamma_2| \geq 12$. If the edge that is touched twice is $e_3$ (marked in gray), then one has the same case as in Figure \ref{fig:PfFig2b} in which $|\gamma_2| \geq 8$. Lastly, if the traversed edge is $e_1$, then the necessary edges that $\gamma_2$ contains are drawn in teal. Here, the smallest closed path that connects those edges contains at least $16$ edges. In particular, these are all possibilities in which $|\gamma_1|=8$. In this case as well we obtained $|\gamma_1|+|\gamma_2| \geq 16$. 
    We proved Lemma \ref{lem:abouttheloopdec} for Case 1, that is, if $\gamma_2$ touches twice an edge $e \in \gamma_1$, then in two dimensions $|\gamma_1|+|\gamma_2| \geq 16$.
    \newline 
\textit{Case 2.} The loop $\gamma_2$ touches at most one boundary of any edge in $e \in \gamma_1$. \newline 
One deduces that in this case, the number of edges in $\gamma_1$ whose boundary is touched by $\gamma_2$ needs to be even. Namely, for a fixed edge $e \in \gamma_1$, then $\gamma_2$ will touch at most one boundary of $e$. If it does, at that boundary $\gamma_2$ will touch the boundary of another edge $\tilde{e} \in \gamma_1$. Hence, the number of edges in $\gamma_1$ whose boundary is touched by $\gamma_2$ needs to be even. Moreover, note that $\gamma_2$ needs to touch the boundary of strictly more than two edges. If it touches the boundary of exactly two edges, as $\gamma_2$ is connected, it simply touches a corner as in Figure \ref{fig:PfFig3a}. The edges in $\gamma_1$ whose boundary is touched by $\gamma_2$ are marked black. In that case, the loop $\gamma=\gamma_1+\gamma_2$ has only one decomposition and doesn't satisfy our criterion of two distinct decompositions. It implies that the loop $\gamma_2$ needs to touch the boundary of at least 4 distinct edges in $\gamma_1$. Hence, we have the following two subcases to observe in order to conclude the proof.
\begin{figure}[H]
    \centering
    \begin{subfigure}[b]{0.45\textwidth}
\begin{tikzpicture}[scale=0.7, every node/.style={font=\small}]
  \draw[step=1cm, very thin, gray!40] (0,0) grid (5,5);

\draw[thick, red] (1,1) rectangle (3,3);
\node[red] at (2,2) {$\gamma_1$};
\draw[thick, blue] (3,3) rectangle (4,4);
\node[blue] at (3.5,3.5) {$\gamma_2$};
\draw[thick,black] (3,3) -- (3,2);
\draw[thick,black] (3,3) -- (2,3);
\draw (3,3) node[fill=black!60,circle,inner sep=1.2pt] {};

  \foreach \x in {0,...,5}
    \foreach \y in {0,...,5}
       \fill[gray!40] (\x,\y) circle (0.03);

\end{tikzpicture}
    \caption{Two loops share exactly one point. The edges in $\gamma_1$ that share a boundary with $\gamma_2$ are marked black.}
    \label{fig:PfFig3a}
    \end{subfigure}
\hfill
\begin{subfigure}[b]{0.45\textwidth}
    \centering

\begin{tikzpicture}[scale=0.7, every node/.style={font=\small}]
  \draw[step=1cm, very thin, gray!40] (0,0) grid (5,5);

\draw[thick, red] (2,2) rectangle (3,3);
\draw[thick, teal] (2,3) -- (1,3);
\draw[thick, teal] (1,3) -- (1,2);
\draw[thick, teal] (1,2) -- (1,1);
\draw[thick, teal] (1,1) -- (3,1);
\draw[thick, teal] (3,1) -- (3,2);
\draw[thick, teal] (3,2) -- (4,2);
\draw[thick, teal] (4,2) -- (4,4);
\draw[thick, teal] (4,4) -- (2,4);
\draw[thick, teal] (2,4) -- (2,3);
\node[red] at (2.5,2.5) {$\gamma_1$};
\node[teal] at (4.5,3.5) {$\gamma_2$};
\draw (2,3) node[fill=black!60,circle,inner sep=1.2pt] {};
\draw (3,2) node[fill=black!60,circle,inner sep=1.2pt] {};

  \foreach \x in {0,...,5}
    \foreach \y in {0,...,5}
      \fill[gray!40] (\x,\y) circle (0.03);

\end{tikzpicture}

    \caption{Configuration where $\gamma_2$ touches two corners of $\gamma_1$ that has minimal size.}
    \label{fig:PfFig3b}
\end{subfigure}
\caption{Figure (a) does not satisfy the criterion of two distinct decompositions in closed loops. Figure (b) shows the minimal configuration in which one loop touches two corner edges of $\gamma_1$.}
\end{figure}
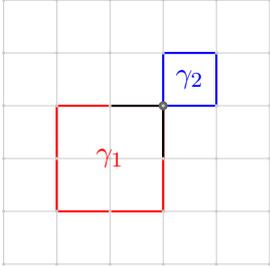
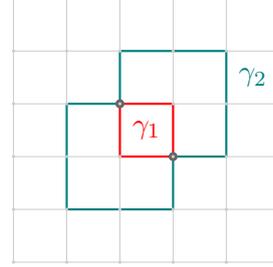

\textit{Case 2a)} The loop $\gamma_2$ only touches corner edges of $\gamma_1$. \newline 
In that case, one loop, let's say $\gamma_2$, needs to touch a corner of $\gamma_1$, then attain another corner of $\gamma_1$, and come back to the first corner to be a closed loop. If $\gamma_1$ has minimal size $|\gamma_1|=4$, then the shortest possible way is $|\gamma_2|=12$ as one can observe in Figure \ref{fig:PfFig3b}. If one tries to make $|\gamma_1|$ larger, then the size of $|\gamma_2|$ cannot decrease as it needs to contain more edges to attain more distant corners of $\gamma_1$. If $\gamma_2$ is the smallest of the two loops, then the above argumentation works by switching the roles of $\gamma_1$ and $\gamma_2$.
It finishes this case.
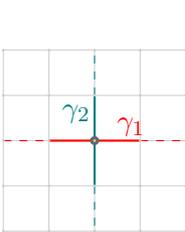
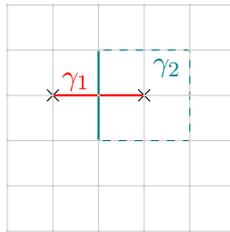
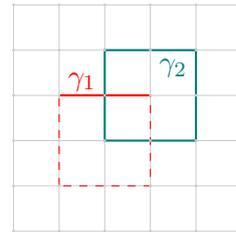
\begin{figure}[H]
    \centering
    \begin{subfigure}[b]{0.32\textwidth}
\begin{tikzpicture}[scale=0.6, every node/.style={font=\small}]
  \draw[step=1cm, very thin, gray!40] (0,0) grid (4,4);

\draw[thick, red] (1,2) -- (3,2);
\draw[dashed, red] (0,2) -- (1,2);
\draw[dashed, red] (3,2) -- (4,2);
\node[red] at (2.8,2.3) {$\gamma_1$};
\draw[thick, teal] (2,1) -- (2,3);
\draw[dashed, teal] (2,1) -- (2,0);
\draw[dashed, teal] (2,3) -- (2,4);
\node[teal] at (1.6,2.6) {$\gamma_2$};
\draw (2,2) node[fill=black!60,circle,inner sep=1.2pt] {};

  \foreach \x in {0,...,4}
    \foreach \y in {0,...,4}
       \fill[gray!40] (\x,\y) circle (0.03);

\end{tikzpicture}
    \caption{$\gamma_2$ crosses a pair of edges in $\gamma_1$.}
    \label{fig:PfFig4a}
    \end{subfigure}
\hfill
\begin{subfigure}[b]{0.32\textwidth}
    \centering

\begin{tikzpicture}[scale=0.6, every node/.style={font=\small}]
  \draw[step=1cm, very thin, gray!40] (0,0) grid (5,5);

\draw[thick, red] (1,3) -- (3,3);
\draw (1,3) node[cross out, draw, inner sep=2pt] {};
\draw (3,3) node[cross out, draw, inner sep=2pt] {};

\node[red] at (1.5,3.3) {$\gamma_1$};
\draw[thick, teal] (2,2) -- (2,4);
\draw[dashed, teal] (2,2) rectangle (4,4);
\node[teal] at (3.5,3.6) {$\gamma_2$};

  \foreach \x in {0,...,5}
    \foreach \y in {0,...,5}
      \fill[gray!40] (\x,\y) circle (0.03);

\end{tikzpicture}

    \caption{$\gamma_2$ crosses a pair of edges in $\gamma_1$, it implies $|\gamma_2| \geq 8$.}
    \label{fig:PfFig4b}
\end{subfigure}
    \begin{subfigure}[b]{0.32\textwidth}
\begin{tikzpicture}[scale=0.6, every node/.style={font=\small}]
  \draw[step=1cm, very thin, gray!40] (0,0) grid (5,5);

\draw[thick, red] (1,3) -- (3,3);
\draw[dashed, red] (1,1) rectangle (3,3);


\node[red] at (1.5,3.3) {$\gamma_1$};
\draw[thick, teal] (2,2) rectangle (4,4);
\node[teal] at (3.5,3.6) {$\gamma_2$};

  \foreach \x in {0,...,5}
    \foreach \y in {0,...,5}
      \fill[gray!40] (\x,\y) circle (0.03);

\end{tikzpicture}

    \caption{The smallest configuration $\gamma_1$ satisfies here $|\gamma_1|=8$.}
    \label{fig:PfFig4c}
    \end{subfigure}
\caption{Representation of the Case 2b) in the proof of Lemma \ref{lem:abouttheloopdec}.}
\end{figure}
\textit{Case 2b)} Assume that $\gamma_2$ doesn't only touches corners of $\gamma_1$, but touches every edge $e$ at at most one boundary. \newline 
Then $\gamma_2$ must cross at least one pair of edges of $\gamma_1$, as pictured in Figure \ref{fig:PfFig4a}. 
By the shape of the two edges of $\gamma_1$ which are crossed by $\gamma_2$, it must hold that $\gamma_1$ contains 6 edges or more, as one observes in Figure \ref{fig:PfFig4a}. On the other hand, it holds that $|\gamma_2| \geq 8$. To see this, observe that if one wants to close the loop $\gamma_2$ in Figure \ref{fig:PfFig4b} using a minimal number of edges, then $\gamma_2$ cannot pass through the crossed points in Figure \ref{fig:PfFig4b}, since it crosses at most one boundary of any edge $e \in \gamma_1$. Therefore, the shortest way to close $\gamma_2$ consists of 8 teal edges as shown in Figure \ref{fig:PfFig4b}. Finally, in that case $\gamma_1$ needs also to contain at least 8 edges to be closed as shown in Figure~\ref{fig:PfFig4c}. It finishes the case in which $|\gamma_2| = 8$. Now, note that if $|\gamma_2| \geq 10$, then one gets $|\gamma_1|+|\gamma_2| \geq 16$ as $|\gamma_1| \geq 6$ from Figure \ref{fig:PfFig4a}. This concludes the proof.
\end{proof}
\newpage
\bibliography{References}

@article{firstpaper,
    author = {Forsström, Malin P and Viklund, Fredrik},
    title = {Free Energy and Quark Potential in Ising Lattice Gauge Theory via Cluster Expansions},
    journal = {International Mathematics Research Notices},
    year = {2025},
}

@book{bookstirling,
    author = {Ronald Graham Donald Knuth and Oren Patashnik},
    title = {Concrete Mathematics: A Foundation for Computer Science},
    year = 1989
}

@article{Shlosman1986SignsOT,
  title={Signs of the Ising model Ursell functions},
  author={Senya B. Shlosman},
  journal={Communications in Mathematical Physics},
  year={1986},
}

@article{Wilson74,
  title = {Confinement of quarks},
  author = {Wilson, Kenneth G.},
  journal = {Phys. Rev. D},
  volume = {10},
  issue = {8},
  pages = {2445--2459},
  numpages = {0},
  year = {1974},
  publisher = {American Physical Society},
}

@article{Forsstr_m_2023,
   title={Wilson loops in the abelian lattice Higgs model},
   volume={4},
   number={2},
   journal={Probability and Mathematical Physics},
   publisher={Mathematical Sciences Publishers},
   author={Forsström, Malin P. and Lenells, Jonatan and Viklund, Fredrik},
   year={2023},
   month=may, pages={257–329} }

@InProceedings{ymprobabilists,
author="Chatterjee, Sourav",
title="Yang--Mills for Probabilists",
booktitle="Probability and Analysis in Interacting Physical Systems",
year="2019",
publisher="Springer International Publishing",
address="Cham",
pages="1--16",
}

@article{chatterjee2020wilsonloopsisinglattice,
      title={Wilson loops in Ising lattice gauge theory}, 
  author       = {Chatterjee, Sourav},
  journal      = {Communications in Mathematical Physics},
  volume       = {377},
  number       = {8},
  pages        = {307--340},
  year         = {2020},
}

@misc{currentexpansion,
      title={A current expansion for Ising lattice gauge theory}, 
      author={Malin P. Forsström and Fredrik Viklund},
      year={2025},
      archivePrefix={arXiv},
      primaryClass={math.PR},
}

@book{ClusteronLGT82,
title = {Gauge Theories as a Problem of Constructive Quantum Field Theory and Statistical Mechanics},
series = {Lecture Notes in Physics},
author = {Seiler, E.},
address = {Berlin, Heidelberg},
publisher = {Springer Berlin Heidelberg},
year = {1982},
edition = {1st ed. 1982},
pages = {195},
}

@book{stmechanics,
place={Cambridge},
title={Statistical Mechanics of Lattice Systems: A Concrete Mathematical Introduction},
publisher={Cambridge University Press},
author={Friedli, Sacha and Velenik, Yvan},
year={2017}
}

@article{Wegner:1971app,
    author = "Wegner, F. J.",
    title = "{Duality in Generalized Ising Models and Phase Transitions Without Local Order Parameters}",
    journal = "J. Math. Phys.",
    volume = "12",
    pages = "2259--2272",
    year = "1971"
}

@book{kogut1979introduction,
  title={An Introduction to Lattice Gauge Theory and Spin Systems},
  author={Kogut, J.B.},
  year={1979},
  publisher={University of Illinois}
}

@article{Elitzur,
  title = {Impossibility of spontaneously breaking local symmetries},
  author = {Elitzur, S.},
  journal = {Phys. Rev. D},
  volume = {12},
  issue = {12},
  pages = {3978--3982},
  numpages = {0},
  year = {1975},
  publisher = {American Physical Society},
}

@article{forsström2021wilsonloopsfiniteabelian,
      title={Wilson loops in finite Abelian lattice gauge theories}, 
author = {Malin P. Forsstr{\"o}m and Jonatan Lenells and Fredrik Viklund},
volume = {58},
journal = {Annales de l'Institut Henri Poincaré, Probabilités et Statistiques},
number = {4},
publisher = {Institut Henri Poincaré},
pages = {2129 -- 2164},
year = {2022},
}

@article{adhikari2024correlationdecayfinitelattice,
  author       = {Arka Adhikari and Sky Cao},
  title        = {Correlation decay for finite lattice gauge theories at weak coupling},
  journal      = {Annals of Probability},
  volume       = {53},
  number       = {1},
  pages        = {140--174},
  year         = {2025},
}

@article{garban2021improvedspinwaveestimatewilson,
 author = {Garban, Christophe and Sepúlveda, Avelio},
    title = {Improved Spin-Wave Estimate for Wilson Loops in U(1) Lattice Gauge Theory},
    journal = {International Mathematics Research Notices},
    volume = {2023},
    number = {21},
    pages = {18142-18198},
    year = {2023},
    month = {01},
}

@misc{cao2025expandedregimesarealaw,
      title={Expanded regimes of area law for lattice Yang-Mills theories}, 
      author={Sky Cao and Ron Nissim and Scott Sheffield},
      year={2025},
      archivePrefix={arXiv},
      primaryClass={math.PR},
}

@article{FrohlichSpencer1982,
  author    = {J{\"u}rg Fr{\"o}hlich and Thomas Spencer},
  title     = {Massless phases and symmetry restoration in abelian gauge theories and spin systems},
  journal   = {Communications in Mathematical Physics},
  volume    = {83},
  number    = {3},
  pages     = {411--454},
  year      = {1982},
}

@article{Guth1980,
  author    = {Alan H. Guth},
  title     = {Existence proof of a nonconfining phase in four-dimensional U(1) lattice gauge theory},
  journal   = {Physical Review D},
  volume    = {21},
  number    = {8},
  pages     = {2291--2307},
  year      = {1980},
}

@article{DROUFFE19831,
title = {Strong coupling and mean field methods in lattice gauge theories},
journal = {Physics Reports},
volume = {102},
number = {1},
pages = {1-119},
year = {1983},
author = {Jean-Michel Drouffe and Jean-Bernard Zuber}
}
\end{document}